\documentclass[12pt,centertags,oneside]{amsart}
\usepackage{amstext,amsthm,amscd,typearea,hyperref}
\usepackage{amsmath,amssymb}
\usepackage{a4wide}
\usepackage[mathscr]{eucal}
\usepackage{mathrsfs}
\usepackage{typearea}
\usepackage{charter}
\usepackage{pdfsync}
\usepackage{xcolor}
\usepackage[a4paper,width=16.2cm,top=3cm,bottom=3cm]{geometry}

\numberwithin{equation}{section}
\usepackage{hyperref}

\newtheorem{theorem}{Theorem}[section]

\newtheorem{proposition}[theorem]{Proposition}
\newtheorem{corollary}[theorem]{Corollary}
\newtheorem{lemma}[theorem]{Lemma}
\newtheorem{remark}[theorem]{Remark}

\newcommand{\ddc}{dd^c}
\newcommand{\dc}{d^c}

\newcommand{\PSH}{{\rm PSH}}

\newcommand{\diam}{\mathop{\mathrm{diam}}\nolimits}
\newcommand{\vol}{\mathop{\mathrm{Vol}}\nolimits}

\newcommand{\capa}{\mathop{\mathrm{Cap}}\nolimits}

\newcommand{\N}{\mathbb{N}}

\newcommand{\R}{\mathbb{R}}

\title[]{\quad Uniform diameter estimates for K\"ahler metrics in big cohomology classes}
\author{Duc-Bao Nguyen and Duc-Viet Vu}
\newcommand{\Addresses}{
{\bigskip
		\footnotesize
		
  \textsc{Duc-Bao Nguyen, National University of Singapore, Department of Mathematics, 10 Lower Kent Ridge Road, 119076, Singapore.}
		\noindent
		\par\nopagebreak
		\noindent
		\textit{E-mail address}: \texttt{ducbao.nguyen@u.nus.edu}}

{
		\bigskip
		\footnotesize
		\textsc{Duc-Viet Vu, University of Cologne, Division of Mathematics, Department of Mathematics and Computer Science, Weyertal 86-90, 50931, K\"oln.}
		\noindent
		\par\nopagebreak
		\noindent
		\textit{E-mail address}: \texttt{dvu@uni-koeln.de}
}}

\begin{document}


\date{\today}

\begin{abstract}
We generalize previous diameter estimates and local non-vanishing of \break volumes for K\"ahler metrics to the case of big cohomology classes. In our proof, among other things, we will prove a uniform diameter estimate for a family of smooth K\"ahler \break metrics only involving an integrability condition. We also have to use fine stability \break properties of complex Monge-Amp\`ere equations with prescribed singularities. 
  \end{abstract}

\medskip

\maketitle

\noindent {\bf Classification AMS 2020}: {32U15},  {32Q15}, {53C23}.

\smallskip
\noindent {\bf Keywords:} {Monge-Amp\`ere equation}, {Gromov-Hausdorff topology}, {closed positive \break current}, {diameter}, {prescribed singularity}, {stability}.

\section{Introduction}

This paper concerns the convergence of K\"ahler spaces in the Gromov-Hausdorff \break topology. In recent years, there has been a growing interest in studying the convergence of K\"ahler metric spaces.  This question arises naturally in several geometric contexts such as the degeneration of K\"ahler-Einstein metrics or the convergence of K\"ahler-Ricci flows,  (see, e.g., \cite{Donaldson-Sun,DonaldsonSun2,Liu-Szekelyhidi,Liu-Szekelyhidi2,Tian-survey-KE,Tosatti-survey,Tosatti-KEflow,Song-Tian-canonicalmeasure}). 

Following the remarkable work \cite{Guo-Phong-Song-Sturm}, a version of Gromov's compactness theorem was established for a very large class of K\"ahler metrics in semi-positive cohomology classes in \cite{Guo-Phong-Song-Sturm2,GPSS_bodieukien,GuedjTo-diameter,Vu-diameter}. A crucial point in these papers is that no assumption about Ricci curvature of metrics in consideration is assumed. Previously, diameter estimates and local non-collapsing of volumes were also established in \cite{Fu-Guo-Song-geometricestimates,Guo-Song-localnoncollapsing}  if Ricci curvature is bounded from below. We refer also to \cite{GGZ-logcontiu,Vu-log-diameter} for some generalizations, and to \cite{Tosatti-collapsing,Rong-Zhang} for the case of Ricci-flat metrics. 

Let $X$ be a compact K\"ahler manifold of dimension $n$ and $\omega_X$ be a fixed K\"ahler form on $X$. Denote by $\mathcal{K}(X)$ the set of smooth K\"ahler metrics on $X$. Let $A,K$ be positive constants and let $p>n$ be a constant. For $\omega \in \mathcal{K}(X)$, we denote
$$V_\omega:= \int_X \omega^n, \quad \mathcal{N}_{X, \omega_X,p}(\omega):= \frac{1}{V_\omega}\int_X \bigg|\log \Big(V_\omega^{-1} \frac{\omega^n}{\omega_X^n} \Big) \bigg|^{p} \omega^n.$$
We define
$$\mathcal{V}(X,\omega_X,n,A,p,K):= \{ \omega \in \mathcal{K}(X): \{\omega\} \cdot \{\omega_X\}^{n-1} \le A, \, \mathcal{N}_{X, \omega_X, p}(\omega) \le K \},$$
where $\{\omega\}, \{\omega_X\}$ denote the cohomology class of $\omega$, $\omega_X$ respectively. If $V_\omega^{-1}\omega^n =f \omega_X^n$ for $f \in L^1 \log L^{p}(X, \omega_X^n)$, then 
$$\mathcal{N}_{X, \omega_X, p}(\omega)=\|f\|_{L^1 \log L^{p}(\omega_X^n)}:= \int_X f |\log f|^{p} \omega_X^n.$$

It was proved in \cite{Guo-Phong-Song-Sturm,GPSS_bodieukien}  (also in \cite{GuedjTo-diameter} or \cite{Vu-diameter}) that the diameter of $(X,\omega)$ for 
$$\omega \in \mathcal{V}(X, \omega_X,n,A,p,K) \quad (p>n)$$
 is uniformly bounded, and a uniform local non-collapsing estimate for volume of $\omega$ holds as well.  In particular, one obtains a version of Gromov's precompactness theorem for the family  $\mathcal{V}(X, \omega_X,n,A,p,K)$ (see \cite[Theorem 1.2]{Guo-Phong-Song-Sturm}). 

 Let $\widetilde{\mathcal{K}}(X)$ be the set of closed positive $(1,1)$-currents $\omega$ of bounded potentials on $X$ such that $\omega$ is a smooth K\"ahler form on an open Zariski dense subset in $X$ and belong to  a big and semi-positive cohomology class in $X$. Recall that a cohomology class is said to be semi-positive if it contains a smooth semi-positive closed form.  For $\omega \in \widetilde{\mathcal{K}}(X)$, the quantity $\mathcal{N}_{X,\omega_X,p}(\omega)$ still makes sense. We put
$$\widetilde{\mathcal{V}}(X,\omega_X,n,A,p,K):= \{ \omega \in \widetilde{\mathcal{K}}(X): \{\omega\} \cdot \{\omega_X\}^{n-1} \le A, \, \mathcal{N}_{X, \omega_X, p}(\omega) \le K \}.$$
This class of metrics was introduced in \cite{Vu-diameter} in order to treat the case of metrics in singular varieties. A slightly more restrictive class of metrics was considered in \cite{Guo-Phong-Song-Sturm2}.   Diameter bounds and local noncollapsing volume estimates were extended to the family of metrics $\tilde{\mathcal{V}}(X,\omega_X,n,A,p,K)$ in \cite{Vu-diameter}. The aim of this paper is to generalize these results to the setting of big cohomology classes. Let us begin with some informal discussion.

Assume that there is a $\mathbb{Q}$-divisor $D$ such that $(X,D)$ is a Kawamata log terminal pair (klt pair for short) and $K_X + D$ is big. Let $T$ be the singular K\"ahler metric solving the equation $\langle T^n \rangle = g \omega_X^n$, where $g$ is a smooth positive function.  
It is known that $T$ is a smooth K\"ahler form on an open Zariski dense subset in $X$ (see e.g. \cite{Guenancia-klt}, and also \cite{BEGZ,Guen-Paun-Bogomolovineqfor3fold,Tsuji,Song-Tian-canonicalmeasure}). The proof of this fact, combined with the uniform diameter estimate for semi-positive forms, implies that the diameter of $T$ is finite. However, using this approach, the obtained diameter bound depends on a log-canonical model of $X$ associated to $K_X+D$. If $D$ varies, then it is not clear if this approach could give a uniform diameter bound for the metric $T$.

Let $T$ be a closed positive $(1,1)$-current which is a smooth K\"ahler form on an open Zariski dense subset $U_T$ in $X$. Let $d_T$ be the metric on $U_T$ induced by $T$. Let $(\widehat X,d_T)$ be the completion of the metric space $(U_T, d_T)$. We extend the smooth measure $T^n$ on $U_T$ to $\widehat X$ by putting $T^n(\widehat X \backslash U_T):=0$. Denote by $V_T$ the total volume $\int_{U_T} T^n$. One observes that $(\widehat X, d_T, T^n)$ is, modulo isometries, independent of the choice of $U_T$. We denote by $\diam(\widehat X, d_T)$ the diameter of $(\widehat X, d_T)$ and by $B_T(x,r)$ the ball of radius $r$ centered at $x \in \widehat X$ in $(\widehat X, d_T)$. Let $\vol_T(B_T(x,r))$ be the volume of $B_T(x,r)$ with respect to measure $T^n$.

Let $p > 1, A,B$ be fixed positive constants. We define the class $\mathcal{W}_{\text{big}}(X,\omega_X,n,p,A,B)$ consisting of  closed positive $(1,1)$-currents $T$ that satisfy the following conditions:
\begin{itemize}
    \item[(i)] $T$ is a smooth K\"ahler form on an open Zariski dense subset $U_T$ of $X$, $T$ belongs to a big cohomology class and $T$ has minimal singularities (see Section \ref{section relative pluripotential theory}  for the notion of minimal singularites);
    \item[(ii)] We have $\{T\} \cdot \{\omega_X\}^{n-1} \leq A\omega_X$, where $\{T\}$ denotes the cohomology class of $T$;
    \item[(iii)] $V_T^{-1} T^n = f_T \omega_X^n$ for some function $f_T$ such that $\|f_T\|_{L^p(\omega_X^n)} \leq B$.
\end{itemize}
As in \cite[Proposition 3.1]{Guo-Phong-Song-Sturm}, the condition (ii) implies that there are a constant $A'>0$ independent of $T$ and a smooth closed $(1,1)$-form $\theta$ cohomologous to $T$ such that $\theta \le A' \omega_X$. We note also that we use the $L^p$-norm in the definition of $\mathcal{W}_{\text{big}}(X,\omega_X,n,p,A,B)$. This is easier for us in computations. We note that when $n\geq 2$, our method also works for $L^1 \log L^p$-norm instead of $L^p$-norms (see Section~\ref{section entropy bound}). 
In what follows, for a constant $C$ we will write $C=C(A_1,A_2,\ldots, A_k)$ if $C$ only depends on given parameters or objects $A_1,A_2,\ldots,A_k$.     
Here is our first main result. 

\begin{theorem}\label{theorem diameter and non collap for big class}
    Let $(X,\omega_X)$ be a compact K\"ahler manifold of dimension $n$. Let $A , B$ be positive constants and $p > 1$ be a constant. Let $q\in \Big(1,\frac{n}{n-1} \Big)$. 
    Then there exist positive constants $C_1= C_1(\omega_X,n,p,A,B)$ and $C_2 = C_2(\omega_X,n,p,A,B,q)$ such that for every $T\in \mathcal{W}_{\text{big}}(X,\omega_X,n,p,A,B)$ we have
     \begin{align} \label{ine-theomain1}
    \diam(\widehat{X},d_T) \leq C_1 \qquad \text{and} \qquad \frac{\vol_T(B_T(x,r))}{V_T} \geq C_2r^{\frac{2q}{q-1}}
    \end{align}
    for every $r\in (0,\diam(\widehat{X},d_T)]$ and $x\in \widehat{X}$. 
\end{theorem}

For $T\in \mathcal{W}_{\text{big}}(X,\omega_X,n,p,A,B)$, recall that $T$ is a K\"ahler form  on an open Zariski dense subset $U_T$. Denote by $W^{1,2}_T(X)$ the set of Borel functions $u$ on $X$ such that $u, \nabla u$ are locally square integrable in $U_T$ and there holds
\[\int_{U_T} |u|^2 T^n < \infty \text{ and } \int_{U_T} du\wedge d^cu \wedge T^{n-1} < \infty.\]
Theorem \ref{theorem diameter and non collap for big class} is a consequence of the following uniform Sobolev inequalities for metrics in $\mathcal{W}_{\text{big}}(X,\omega_X,n,p,A,B)$. This extends Sobolev inequalities obtained in \cite{Guo-Phong-Song-Sturm,Guo-Phong-Song-Sturm2}; see also \cite{GuedjTo-diameter}.

\begin{theorem}\label{theorem Sobolev for big case}
Let $1<r<\frac{n}{n-1}$ be a constant. Let $T\in \mathcal{W}_{\text{big}}(X,n,p,A,B)$. Then there exists a positive constant $C = C(\omega_X,n,p,A,B,r)$ such that 
\begin{itemize}
        \item[(i)] 
        \[\Big(\frac{1}{V_T} \int_{U_T} |u - \overline{u}|^{2r} T^n \Big)^{\frac{1}{r}} \leq C \Big(\frac{1}{V_T} \int_{U_T} du\wedge d^c u \wedge T ^{n-1}\Big)\]
        for every $u\in W^{1,2}_T(X)$, where $\overline{u}:= V_T^{-1} \int_{U_T} u T^n$,
        \item[(ii)] 
    \[\Big(\frac{1}{V_T} \int_{U_T} |u|^{2r} T^n \Big)^{\frac{1}{r}} \leq C \Big( \frac{1}{V_T} \int_{U_T} du\wedge d^c u \wedge T^{n-1} + \frac{1}{V_T} \int_{U_T} |u|^2 T^n\Big)\]
        for every $u\in W^{1,2}_T(X)$.
    \end{itemize}    
\end{theorem}

We obtain immediately the following consequence of Theorem \ref{theorem diameter and non collap for big class}.

\begin{theorem} \label{the-kltpair}
    Let $(X,\omega_X)$ be a compact K\"ahler manifold of dimension $n$. Let $q\in \Big(1,\frac{n}{n-1} \Big)$.  Let $D$ be a $\mathbb{Q}$-divisor such that $(X,D)$ is a klt pair and $K_X + D$ is big. Let $g$ be a smooth positive function such that $ \int_X g \omega_X^n = 1$. Assume that 
    $$\{D\} \leq A\{\omega_X\}, \quad \|g\|_{L^p(\omega_X^n)} \leq B$$
    for some positive constants $A,B$, where $\{D\}$ denotes the cohomology class of the current of integration along $D$. Let $T$ be the closed positive $(1,1)$-current in $K_X +D$ satisfies $$V_T^{-1}\left\langle T^n \right\rangle = g \omega_X^n,$$
    where $\langle T^n \rangle$ denotes the non-pluripolar self-product of $T,\ldots, T$ (n times). Then there exist constants $C_1= C_1(\omega_X,n,p,A,B)$ and $C_2 = C_2(\omega_X,n,p,A,B,q)$ such that  (\ref{ine-theomain1}) holds.  
\end{theorem}

We recall here that the current $T$ as in Theorem \ref{the-kltpair} must belong to $\mathcal{W}_{\text{big}}(X,\omega_X,n,p',A',$ $B')$ for some constants $p'>1,A',B'>0$  by \cite{Guenancia-klt}, see also \cite{BEGZ} and \cite{Guen-Paun-Bogomolovineqfor3fold}.

Let $\mathcal{M}(X,\omega_X,n,p,A,B)$ be the set of all $(\widehat{X},d_T)$ for $T \in \mathcal{W}_{\text{big}}(X,\omega_X,n,p,A,B)$. Note that since $\diam(\widehat{X},d_T)$ is finite, $(\widehat{X},d_T)$ is a compact metric space. Thus, by Theorem~\ref{theorem diameter and non collap for big class} and the proof of \cite[Theorem 1.2]{Guo-Phong-Song-Sturm}, we obtain immediately the following corollary.

\begin{corollary}
    Let $p>1,A,B$ be positive constants. Then the set $\mathcal{M}(X,\omega_X,n,p,A,B)$ is relatively compact in the Gromov-Hausdorff topology on the space of compact metric spaces up to isometry.
\end{corollary}

Family versions of the above results also hold. Let $\mathcal{X}$ be an irreducible and reduced complex K\"ahler space. Let $\pi: \mathcal{X}\rightarrow 2\mathbb{D}$ be a proper, surjective holomorphic map with connected fibers such that each fiber $X_t = \pi^{-1}(t)$ is a $n$-dimensional compact K\"ahler manifold for $t\neq 0$ and $X_0$ is an irreducible and reduced complex K\"ahler space. Let $\omega_{\mathcal{X}}$ be a fixed K\"ahler form on $\mathcal{X}$ and put $\omega_t := \omega_{\mathcal{X}}|_{X_t}$. Note that the total volumes $V_{\omega_t} = \int_{X_t} \omega_t^n$ are independent of $t$ by \cite[Lemma 2.2]{DiNezzaGG}. We then obtain the following family version of Theorem~\ref{theorem diameter and non collap for big class}.

\begin{theorem}\label{theorem geometric for family}
Let $p >1,A,B$ be positive constants. Let $q\in \Big(1,\frac{n}{n-1} \Big)$. Then there exist constants $C_1 = C_1(n,p,A,B)$ and $C_2 = C_2 (n,p,A,B,q)$ such that
    \begin{align} \label{ine-theomain12}
    \diam(\widehat{X}_t,d_{T_t}) \leq C_1 \qquad \text{and} \qquad \frac{\vol_{T_t} (B_{T_t}(x,r))}{V_{T_t}}\geq C_2 r^{\frac{2q}{q-1}}
    \end{align}
    for every $T_t \in \mathcal{W}_{\text{big}}(X_t,\omega_t,n,p,A,B)$, $t\in \mathbb{D}^*$, $r \in (0,\diam(\widehat{X}_t,d_{T_t})]$, and $x\in \widehat{X}_t$.
\end{theorem}

We now comment on the proof of Theorem \ref{theorem diameter and non collap for big class}.  We start with a naive idea to try to approximate a metric in a big cohomology class by those in K\"ahler classes. A standard way to do it is to use Demailly's analytic approximation of quasi-plurisubharmonic functions (qpsh for short) and pass to desingularization to obtain metrics in semipositive big classes. Precisely let $T$ be a metric as in the statement of Theorem \ref{theorem diameter and non collap for big class}, we approximate $T$ by a sequence of closed positive currents $T_k$ with analytic singularities. Let $\pi_k: X_k \to X$ be a smooth modification so that $\pi^*_k T_k= \omega_k+ D_k$, where $\omega_k$ is a smooth semi-positive form on $X_k$ and $D_k$ is an effective divisor. This gives us hope to apply known results in \cite{Guo-Phong-Song-Sturm,GuedjTo-diameter,Vu-diameter} to $\omega_k$. However, this is not obvious because of the following two issues. First, it is not clear if we can control the density of the measure $(T_k)^n$ (thus no control on $\omega_k^n$). Secondly, even if we can solve the first issue in an appropriate way, previous results on diameter estimates in \cite{Guo-Phong-Song-Sturm,GuedjTo-diameter,Vu-diameter} don't apply directly. The reason is that the class of metrics in these aforementioned references depends crucially on a fixed reference K\"ahler form on the ambient manifold. Whereas we are dealing with a set of smooth modifications $(X_k)_k$ which is not of types considered in \cite{GuedjTo-diameter}. 

To remedy these problems, we first establish a refinement of Green's function estimates for smooth K\"ahler metrics which involves only an integrability condition on the Monge-Amp\`ere measure of the metrics. The proof of this mainly follows those in \cite{Guo-Phong-Song-Sturm,GuedjTo-diameter}. Our contribution here is to realize that no reference K\"ahler form is needed to deduce estimates on Green's function.  

The next part of our proof is to solve the first issue. To this end, using \cite{Lu-Darvas-DiNezza-logconcave}, we obtain a modified metric $\widetilde{T}_k$ of the same singularity type as $T_k$ which satisfies a suitable Monge-Amp\`ere equation. It is not difficult to obtain the weak convergence of $\widetilde{T}_k$ to $T$ by standard techniques. However, we will need a stronger property which is the convergence in capacity of potentials of $\widetilde{T}_k$ to that of $T$. This is not trivial and follows from the stability of Monge-Amp\`ere equations with prescribed singularities proved in \cite{Vu_DoHS-quantitative1,Vu_DoHS-quantitative} (also \cite{Darvas-Lu-DiNezza-singularity-metric}), see Theorem \ref{theorem stability for MA with prescribed singularity} below. Finally, in order to ensure the validity of the integrability condition in the above-mentioned refinement of Green's function estimates,  we define and study the notion of good dsh functions. This is a subclass of dsh functions introduced in \cite{DS_tm}.  Good dsh functions behave somewhat as qpsh functions but they are more general and more suitable to deal with K\"ahler metrics with singularity.  
\\

\noindent \textbf{Aknowledgment.} The first named author is supported by the Singapore International Graduate Award (SINGA). The second named author is partially supported by the \break Deutsche Forschungsgemeinschaft (DFG, German Research Foundation)-Projektnummer \break500055552 and by the ANR-DFG grant QuaSiDy, grant no ANR-21-CE40-0016. We would like to thank Tien-Cuong Dinh for his support during the preparation of this paper.

\section{Uniform estimates for subharmonic functions}\label{section uniform for Laplace}
Our first result in this section is the following refinement of the main result in \cite{Guedj-Lu-1}.

\begin{theorem}\label{theorem - uniform estimate Guedj-Lu condition}
    Let $\mu$ be a probability measure on $X$ and let $\omega$ be a semi-positive big form. Assume that there exist constants $m>n$ and $A>0$ such that
    \[\Big(\int_X (-\psi)^m d\mu\Big)^{\frac{1}{m}} \leq A\]
    for every $\omega$-psh function $\psi$ with $\sup_X \psi = 0$. Let $\varphi$ be the unique solution of the equation
    \[V_\omega^{-1}(\omega + \ddc \varphi)^n= \mu \text{ with } \sup_X \varphi =0.\]
    Then
    \[\|\varphi\|_{L^\infty} \leq C (1+A^{\frac{3}{n+2m-3}}),\]
    where $C = C(m,n)$.
\end{theorem}

Recall that an $\omega$-psh function $\psi$ is a qpsh function on $X$ such that $\omega + dd^c \psi \geq 0$. 

\begin{proof} The proof is essentially that of \cite[Theorem 2.1]{Guedj-Lu-1}. We just need to compute constants explicitly.
    Let $T_{\max} : = \sup \{t > 0: \mu(\varphi < -t) > 0\}$. It is known that $T_{\max}$ is finite. Our goal is to get an explicit bound for $T_{\max}$. Let $\varepsilon = \frac{m-n}{3}\cdot$ Fix $T_{\max}< T_0 \le 0$. We define $g:\mathbb{R}^+ \rightarrow \mathbb{R}^+$ by $g(0) = 1$ and 
    \[ g'(t) := \left\{\begin{matrix}
\frac{1}{(1+t)^2  \mu(\varphi <-t)} & \text{if }t\leq T_0, \\
 \frac{1}{(1+t)^2 }&   \text{if }t>T_0.
\end{matrix}\right.\]
Choose $\chi$ such that $\chi(0) = 0$, $\chi'(0) = 1$ and $\chi'(-t) = g(t)^{\frac{1}{n+2\varepsilon}}$. Put  $h:= -\chi(-t)$. 
In what follows we will use $\lesssim, \gtrsim$ to denote $\le, \ge$ modulo a multiplicative constant depending only on $m,n$.  As in the proof of \cite[Theorem 2.1]{Guedj-Lu-1}, for $t\in [0,T_0]$, we have
\begin{align}\label{eq control for h }
\frac{h^{m+1}(t)}{m+1} \lesssim A(1+t)^2 \cdot (h')^{n+2\varepsilon +1} (t)
\end{align}
for every $t \in [0,T_0]$. Rewrite \eqref{eq control for h }, we get
\begin{equation}\label{eq control for h 2}
(1+t)^{-\frac{2}{n+2\varepsilon +1 }} \lesssim A^{\frac{1}{n+2\varepsilon +1}} h'(t) h^{-\frac{m+1}{n+2\varepsilon +1}}.
\end{equation}
Integrating from $1$ to $T_0$ for both side of \eqref{eq control for h 2}, we get
\begin{align}\label{eq control for h 3}
\frac{(1+T_0)^{1 -\frac{2}{n+2\varepsilon +1 }} - 2^{1-\frac{2}{n+2\varepsilon +1 }}}{1-\frac{2}{n+2\varepsilon +1 }} \lesssim A^{\frac{1}{n+2\varepsilon +1}} \Big(\frac{h(T_0)^{1-\frac{m+1}{n+2\varepsilon +1}} - h(1)^{1-\frac{m+1}{n+2\varepsilon +1}}}{1-\frac{m+1}{n+2\varepsilon +1}} \Big) \cdot   
\end{align}
Note that $2< n+2\varepsilon + 1,m+1 >  n+2\varepsilon + 1$ and $h(1) \geq 1$, from \eqref{eq control for h 3}, we infer that
\begin{equation}\label{eq final bound}
    T_0 \lesssim 1 + A^{\frac{3}{n+2m-3}}
\end{equation}
Since $T_0$ is arbitrary in $[0,T_{\max}]$, we get our desired estimate.
\end{proof}

From now on, let $m > n, A> 0$ be  constants such that
\begin{equation}\label{eq cond for m}
    \tag{*} \lambda := \frac{3}{m(n+2m-3)} < \frac{n}{n+1}\cdot
\end{equation}
Note that this condition is always true if $n\geq 2$.
We define the class $\mathcal{V}_{\mathcal{K}}(X,n,m,A)$ be the set of smooth K\"ahler forms $\omega$ such that 
    \begin{equation}\label{cond V}
    \frac{1}{V_\omega}\int_X (-\psi)^m \omega^n \leq A
    \end{equation}
    for every $\omega$-psh function $\psi$ with $\sup_X \psi = 0$. Our goal is to show that the metrics in $\mathcal{V}_{\mathcal{K}}(X,n,m,A)$ satisfy a uniform diameter bound and a uniform local-noncollapsing volume estimate as in \cite{Guo-Phong-Song-Sturm}.  We follow more or less the same strategy as in \cite{Guo-Phong-Song-Sturm,GuedjTo-diameter}. 
We first obtain the following version of \cite[Lemma 2]{Guo-Phong-Sturm-green} and \cite[Lemma 5.1]{Guo-Phong-Song-Sturm}.

\begin{lemma}\label{lemma Delta>-a for omega} Let $\omega \in \mathcal{V}_\mathcal{K}(X,n,m,A)$. Let $a>0$ be a fixed constant and $v$ be a real integrable function on $X$. For $s\in \R$, put $\Omega_s := \{v > s \}$ is the upper-level set of $v$. Assume that $\Omega_{-1}$ is open. Suppose that $v$ is a $\mathscr{C}^2$ function on an open neighborhood $U$ of $\overline{\Omega}_{-1}$,
    \[\Delta_{\omega} v \geq -a \text{ on } \Omega_0 \text{ and }\int_X v  \omega^n = 0.\]
    Then
    \[\sup_X v \leq C \Big(a + \frac{1}{V_{\omega}} \int_X |v|  \omega^n \Big),\]
    where $C = C(n,m,A) $.
\end{lemma}

\begin{proof} We follow closely the proof of \cite[Lemma 2.1]{GuedjTo-diameter} together with the uniform estimate in the Theorem~\ref{theorem - uniform estimate Guedj-Lu condition}. 

By homogeneity, we only need to prove for $a =n$. We approximate $v$ as in \cite{Guo-Phong-Song-Sturm}. Consider a sequence of smooth functions $\eta_k: \R \rightarrow \R_{\geq 0}$ such that $\eta_k$ converges uniformly and decreasingly to the function $x \cdot \mathbf{1}_{\R_+}(\cdot)$. We choose $\eta_k(x) \equiv 1/k$ for $x\leq -1/2$ and $\eta_k(x) \leq 1/k$ for $ x \in (-1/2,0)$. Put $v_k:= \eta _k \circ v$. Since $v \in \mathscr{C}^2$ on $U$, $v_k $ is $\mathscr{C}^2$ on $U$. Let $x\notin \overline{\Omega}_{-1}$, then there exists an open neighborhood $V$ of $x$ such that $V\cap \Omega_{-1} = \emptyset$. Thus, $v_k \equiv 1/k$ on $V$ and then $\mathscr{C}^2$ on $V$. We have just proved that $v_k$ is $\mathscr{C}^2$ on $X$. We note also that $v_k = v$ on $\Omega_0$.

Consider the equation
\begin{equation}\label{eq varphi for Laplacian 1}(\omega +dd^c \varphi)^n =  \frac{1+v_k}{1+M_k} \cdot \omega^n \text{ with } \sup_X \varphi = -1,\end{equation}
where $M_k= V_\omega^{-1} \int_X v_k \omega^n$. Since the right-hand side of \eqref{eq varphi for Laplacian 1} is positive and $\mathscr{C}^2$, there exists a unique $\mathscr{C}^2$ solution $\varphi \in \PSH(X,\omega)$ by Yau's theorem \cite{Yau1978}.

We define $H := 1+ v_k - \varepsilon (-\varphi)^\alpha$ where 
\begin{equation}\label{eq choose-epsilon}\alpha = \frac{n}{n+1} \text{ and } \varepsilon \text{ satisfies }\frac{\alpha^n \varepsilon^{n+1}}{(1+\alpha \varepsilon)^n} = 1+M_k. \end{equation}
Note that from this formula, $\varepsilon \geq 1$. Arguing exactly as in the proof of \cite[Lemma 5.1]{Guo-Phong-Song-Sturm}, we obtain that $H\leq 0$. Hence, we get 
\begin{equation}\label{eq compare v+ varphi}
    1+v\leq \varepsilon (-\varphi)^\alpha \text{ on }X.
\end{equation}

It remains to bound $\varphi$. Let $f =  \frac{1+v_k}{1+M_k}$ and $\mu:= f \omega^n$. Recall that $V_{\omega}^{-1}(\omega+dd^c \varphi)^n = \mu$.
By (\ref{cond V}), we see that 
$$    \frac{1}{V_\omega}\int_X (-\psi)^m \omega^n \leq A \|f\|_{L^\infty}^{1/m}$$
    for every $\omega$-psh function $\psi$ with $\sup_X \psi = 0$. This coupled with Theorem \ref{theorem - uniform estimate Guedj-Lu condition} gives 
    \[\|\varphi\|_{L^\infty} \leq C' \|f\|_{L^\infty}^\lambda,\]
where $C' = C'(n,m,A)$. Note that $v_k \leq 1/k$ outside $\Omega_0$ and $v_k = v$ on $\Omega_0$. Thus, by \eqref{eq compare v+ varphi}, we infer that
\begin{equation}\label{eq compute L infty varphi}
    \|\varphi\|_{L^\infty} \leq  C'(1+M_k)^{-\lambda} \varepsilon^{\lambda} \|\varphi\|_{L^\infty}^{\alpha \lambda}.
\end{equation}
Since $m$ satisfies condition~\eqref{eq cond for m}, we have $\alpha \lambda <1$. Hence, 
\begin{equation}\label{eq compute L infty varphi 2}
    \|\varphi\|_{L^\infty} \leq \Big(C' (1+M_k)^{-\lambda} \varepsilon^{\lambda}  \Big)^{(1-\alpha\lambda)^{-1}}.
\end{equation}

By \eqref{eq compare v+ varphi} again and note that $\varepsilon \leq c_n (1+M_k)$ where $c_n$ is a constant depending on the dimension $n$, we have
\begin{align*}1+\sup_X v &\leq \varepsilon \Big(C' (1+M_k)^{-\lambda} \varepsilon^{\lambda}  \Big)^{\alpha(1-\alpha\lambda)^{-1}} \\ \nonumber
&\leq c_n  (C' c_n^\lambda)^{\alpha(1-\alpha\lambda)^{-1}} (1+M_k).
\end{align*}
We choose 
\begin{equation*}
    C := c_n  (C' c_n^\lambda)^{\alpha(1-\alpha\lambda)^{-1}}.
\end{equation*}
Then $C = C(n,m,A)$. Let $k\rightarrow \infty$, then $M_k \rightarrow (2V_\omega)^{-1} \int_X |v|\omega^n$. We complete the proof.
\end{proof}

The following estimate is similar to \cite[Proposition 2.2]{GuedjTo-diameter} and \cite[Proposition 2.1]{GPSS_bodieukien}.

\begin{proposition}\label{prop |delta|<=1 Linfty for v}
Let $\omega \in \mathcal{V}_\mathcal{K}(X,n,m,A)$. Let $v$ be a $\mathscr{C}^2$ function on $X$ such that 
    \[|\Delta_{\omega} v| \leq 1 \quad \text{and} \quad \int_X v\omega^n = 0.\]
    Then
    \[\|v\|_{L^\infty} \leq C,\]
    where $C = C(n,m,A)$.
\end{proposition}

\begin{proof} The proof is identical to that of \cite[Proposition 2.1]{GPSS_bodieukien}. We just need to use Theorem \ref{theorem - uniform estimate Guedj-Lu condition} instead of the $L^\infty$-estimate used in \cite{GPSS_bodieukien}. 
\end{proof}

\section{Uniform estimates for Green's functions}\label{section uniform for Green}

In this section, we prove uniform estimates for Green's function of K\"ahler metric $\omega $ in $ \mathcal{V}_\mathcal{K}(X,n,m,A)$. Recall that the normalized Green's function $G_\omega(x,y)$ associated to $(X,\omega)$ is defined by
\[ \frac{1}{V_\omega} (\omega + dd^c G_\omega(x,\cdot)) \wedge \omega^{n-1} = \delta_x \text{ with } \int_X G_\omega(x,\cdot) \omega^n = 0.\]
Here $\delta_x$ denotes the Dirac mass at $x$. This function has been considered in \cite{GuedjTo-diameter}. Note that if we let $\widetilde{G}_\omega(x,\cdot)$ be the Green's function in \cite{Guo-Phong-Song-Sturm}, $G_\omega(x,\cdot) = - \frac{V_\omega}{n} \widetilde{G}_\omega(x,\cdot)$. It is known that $G_\omega(x,y)$ always exists, is smooth on $X\times X$ off the diagonal, and is symmetric. We note also that $G_\omega(x,y)$ satisfies the asymptotic behavior (see \cite[Chapter 5. Theorem 3.5]{Yau-Schoen-LectureDiffgeo})
\begin{equation}\label{eq aysmptotic behavior}G_\omega(x,y) \sim d_\omega (x,y)^{-2n+2} \text{ if } n\geq 2  \text{ and } G_\omega(x,y) \sim \log d_\omega (x,y) \text{ if } n = 1.\end{equation}
For any smooth function $u$, recall also the Green's formula
\[u(x) = \frac{1}{V_\omega} \int_X u (\omega + dd^c G_\omega(x,\cdot)) \wedge \omega^{n-1} = \frac{1}{V_\omega} \int_X G_\omega(x,\cdot) (\omega + dd^c u) \wedge \omega^{n-1}.\]

Here is the main result of this section.

\begin{theorem}\label{theorem uniform Green semipositive class}
Let $\omega \in \mathcal{V}_\mathcal{K}(X,n,m,A)$ be a K\"ahler metric. For 
    \[0< r < \frac{m-n}{mn-m+n} \text{ and } 0<s<\frac{m-n}{2mn - m + n} \ ,\]
     there exist positive constants $C_1 = C_1(n,m,A),C_2 = C_2(n,m,A,r), C_3 = C_3(n,m,A,s)$ such that
    \begin{itemize}
        \item[(i)] $\sup\limits_{y\in X} G_{\omega}(x,y) \leq C_1 $ ,
        \item[(ii)] $\displaystyle \frac{1}{V_\omega} \int_X |G_\omega(x,\cdot)|^{1+r} \omega^n\leq C_2 $ ,
        \item[(iii)] $\displaystyle \frac{1}{V_\omega}\int_X |\nabla G_{\omega}(x,\cdot)|^{1+s}_\omega \omega^n\leq C_3$ ,
    \end{itemize}
    for every $x\in X$.
\end{theorem}
 We follow the proof in \cite{Guo-Phong-Sturm-green,Guo-Phong-Song-Sturm,GuedjTo-diameter} together with the uniform estimates in Section~\ref{section uniform for Laplace}. To lighten the notation, we will write $G_x$ instead of $G_\omega(x,\cdot)$.

\begin{proof}[Proof of (i)]
Let $\omega$ be a fixed K\"ahler metric in $\mathcal{V}_\mathcal{K}(X,n,m,A)$ and $x$ be a fixed point in $X$. Consider the function
    \[h = -\mathbf{1}_{\{G_x\leq 0\}} +  \frac{1}{V_\omega}\int_{\{G_x\leq 0\}}\omega^n\cdot\]
Approximating $h$ in $L^q(X,\omega^n)$ (for some $q>0$ big) by a sequence of smooth functions $h_k$ such that
\[\sup_X |h_k| \leq 2 \text{ and } \frac{1}{V_\omega} \int_{X} h_k \omega^n = 0.\]
Let $v_k$ be the solution of the equation
    \[\Delta_{\omega} v_k = h_k \quad \text{with} \quad \int_X v_k \omega^n = 0.\]
    By Proposition \ref{prop |delta|<=1 Linfty for v}, $\|v_k\|_{L^\infty} \leq C$ where $C=C(n,m,A)$. Applying Green's formula to smooth function $v_k$ and note that $\int_X v_k \omega^n = 0$, we have
    \[v_k(x) = \frac{1}{V_\omega}\int_X G_x \cdot \frac{h_k}{n} \omega^n.\]
    By \eqref{eq aysmptotic behavior} and H\"older's inequality, we can choose $q$ big enough such that 
    \[\frac{1}{V_\omega}\int_X G_x \cdot \frac{h_k}{n} \omega^n \rightarrow  \frac{1}{V_\omega}\int_X G_x\cdot \frac{h}{n}\omega^n\]
    as $k\rightarrow \infty$. This implies that
    \[\frac{1}{V_\omega}\int_{\{G_x \leq 0\}} G_x \omega^n \leq C.\]
    Moreover, since $\int_X G_x\omega^n = 0$, we deduce that
    \[\frac{1}{V_\omega}\int_X G_x \omega^n \leq C.\]
    Thus, applying Lemma \ref{lemma Delta>-a for omega} to $G_x$ (this function is smooth on $X\setminus \{x\}$ and goes to infinity at $x$), we get
    \[G_x \leq C' \Big(1 + \frac{1}{V_\omega} \int_X G_x \omega^n\Big) \leq C'(1+C)=: C_1 \]
    for $C' = C'(n,m,A)$ (and thus $C_1 = C_1 (n,m,A)$). We complete the proof for (i).
\end{proof}

\begin{proof}[Proof of (ii)]

Define $\mathcal{G}_x:= G_x - C_1 - 1 \leq -1$. Let $\beta \in \Big(0,\frac{m-n}{mn}\Big)$ be a constant. Approximating $\mathcal{G}_x$ in $L^{1+\beta}(X,\omega^n)$ by a sequence of smooth funtions $\mathcal{G}_k$. We note that here we use the asymptotic behavior \eqref{eq aysmptotic behavior} (in fact, we can approximate $\mathcal{G}_x$ in $L^q$ norm for any $q < \frac{n}{n-1}$). 

Let $v_k$ be the solution of the equation
\begin{equation}\label{eq Laplace for v_k}\frac{1}{V_{\omega}}(\omega + dd^c v_k) \wedge \omega^{n-1} = \frac{(-\mathcal{G}_k)^{\beta} \omega^n}{\int_X (-\mathcal{G}_k)^{\beta} \omega^n} \text{ with }\int_X v_k \omega^n = 0 .\end{equation}
Since $v_k$ is a smooth function, by Green's formula, we have 
\begin{align*} v_k(x) 
&= \frac{1}{V_{\omega}} \int_X \mathcal{G}_x dd^c v_k \wedge \omega^{n-1} \\ 
&= \frac{1}{V_{\omega}} \int_X \mathcal{G}_x (\omega+dd^c v_k) \wedge \omega^{n-1} - \frac{1}{V_{\omega}} \int_X \mathcal{G}_x \omega^{n} \\ 
&= - \frac{ V_\omega^{-1}\int_X (-\mathcal{G}_x) \cdot (-\mathcal{G}_k)^{\beta} \omega^n}{ V_\omega^{-1}\int_X (-\mathcal{G}_k)^{\beta}\omega^n} - \frac{1}{V_{\omega}} \int_X \mathcal{G}_x\omega^{n} \cdot
\end{align*}
This is equivalent to
\begin{equation}\label{eq v_k Green}\Big(v_k + \frac{1}{V_\omega}\int_X \mathcal{G}_x \omega^n \Big) \cdot \Big(\frac{1}{V_\omega} \int_X (-\mathcal{G}_k)^\beta\Big) = - \frac{1}{V_\omega}\int_X (-\mathcal{G}_x) \cdot (-\mathcal{G}_k)^{\beta} \omega^n.\end{equation}
Since $\mathcal{G}_x \leq -1$, we can bound
\[\frac{1}{V_\omega}\int_X (-\mathcal{G}_x)^{\beta}\omega^n \leq \frac{1}{V_\omega} \int_X (-\mathcal{G}_x)\omega^n = 1 +C_1.\]
Letting $k\rightarrow \infty$ in \eqref{eq v_k Green} gives
\begin{align*} \frac{1}{V_\omega}\int_X |G_x|^{1+\beta}\omega^n 
&\leq (1+C_1)^2(1+\limsup\limits_{k\rightarrow \infty}\|v_k\|_{L^\infty}).
\end{align*}

It remains to bound $v_k$. Let $\varphi \in \PSH(X,\omega)$ be the unique solution of the equation
\begin{equation}\label{eq varphi for Green}
V_{\omega}^{-1}(\omega+dd^c \varphi)^n = \frac{(-\mathcal{G}_k)^{n\beta} \omega^n}{\int_X(-\mathcal{G}_k)^{n\beta} \omega^n} \text{ with } \sup_X \varphi = 0.\end{equation} Let $\psi$ be an $\omega$-psh function with $\sup_X \psi = 0$. By H\"older inequality, we have
\begin{align}\label{eq control Lm of mu for Green}\frac{1}{V_\omega}\int_X (-\psi)^{m(1-n\beta)}(-\mathcal{G}_k)^{n\beta} \omega^n &\leq \Big(\frac{1}{V_\omega}\int_X (-\mathcal{G}_k) \omega^n\Big)^{n\beta} \Big(\frac{1}{V_\omega} \int_X (-\psi)^m \omega^n\Big)^{1-n\beta} \\ \nonumber &\leq (2+C_1)^{n\beta} A^{m(1-n\beta)}.\end{align}
Again we consider $k$ big enough. Since $\beta < \frac{m-n}{mn}$, we get $m(1-n\beta) > n$. Thus \eqref{eq control Lm of mu for Green} combines with Theorem~\ref{theorem - uniform estimate Guedj-Lu condition} yields
\begin{equation}\label{eq Linfty varphi <= M for green}\|\varphi\|_{L^\infty} \leq M,\end{equation}
where $M=M(n,m,A,\beta)$. Note that here since $\mathcal{G}_x\leq -1$, we can assume $\mathcal{G}_k \leq -1$ and bound the denominator of \eqref{eq varphi for Green} below by $1$.

Bounding $V_\omega^{-1}\int_X (-\mathcal{G}_k)^{n\beta} \omega^n \leq 2+C_1 $ and using AM-GM inequality (see \cite[Theorem 2.12]{Dinew-Kolodzeij-AprioriHessian}), we have
\begin{equation}\label{compare u rho u}
    (\omega + dd^c \varphi)\wedge \omega^{n-1} \geq \frac{(-\mathcal{G}_k)^\beta \omega^{n}}{(2+C_1)^{\frac{1}{n}}} \cdot
\end{equation}
Moreover, since $\mathcal{G}_k \leq -1$, using \eqref{eq Laplace for v_k}, we get 
\[(-\mathcal{G}_k)^\beta \omega^{n} \geq (\omega + dd^c v_k) \wedge \omega^{n-1}.\]
Thus, if we put
\[v: = \varphi - \frac{v_k}{(2+C_1)^{\frac{1}{n}}} - \frac{1}{V_\omega}\int_X \varphi \omega^n,\]
then $\Delta_\omega v \geq -n$ and $\int_X v \omega^n = 0$. Applying Green's formula to smooth function $v$, we get
\begin{equation}\label{eq bound for v Green}
    v(x) \leq \frac{1}{V_\omega} \int_X (-\mathcal{G}_x) \omega^n = 1+C_1.
\end{equation} 
Combining \eqref{eq Linfty varphi <= M for green} and \eqref{eq bound for v Green}, we get the bound for $v_k$. The proof for (ii) when $r < \frac{m-n}{mn}$ follows.

Now, for $k\in \N$, assume that we have the uniform bound for $ V_\omega^{-1}\int_X (-\mathcal{G}_x)^{1+r_k} \omega^n $ for some $r_k \geq 0$. Arguing similarly to the above proof and replacing \eqref{eq control Lm of mu for Green} by the following estimate:
\[
    \frac{1}{V_\omega}\int_X (-\psi)^{m(1-\frac{n\beta}{1+r_k})}(-\mathcal{G}_k)^{n\beta} \omega^n \leq \Big(\frac{1}{V_\omega}\int_X (-\mathcal{G}_k)^{1+r_k} \omega^n\Big)^{n\beta} \Big(\frac{1}{V_\omega} \int_X (-\psi)^m \omega^n\Big)^{1-\frac{n\beta}{1+r_k}},
\]
we can prove the uniform bound for $V_\omega^{-1}\int_X (-\mathcal{G}_x)^{1+r} \omega^n$ with
\[r < r_{k+1}:=\frac{m-n}{mn} \cdot (1+r_k)\cdot\]
Put $r_0:= 0$. Note that 
\[\lim\limits_{k\rightarrow \infty} r_k = \frac{m-n}{mn-m+n}\cdot\]
By induction, we obtain the uniform bound for 
\[r < \frac{m-n}{mn-m+n}\cdot\]
We complete the proof for (ii). 
\end{proof}

\begin{proof}[Proof of (iii)]

Let $\beta > 0$ small. Consider the function $u(y):= (-\mathcal{G}_x (y))^{-\beta}$. Then $u(y)$ is a continuous function and smooth on $X \setminus \{x\}$ with $u(x) = 0$.
We have
\[ d u = \frac{\beta d\mathcal{G}_x}{(-\mathcal{G}_x)^{1+\beta}}\cdot\]
We now use the asymptotic behavior of $\mathcal{G}_x$ in \cite[Chapter 5. Theorem 3.5]{Yau-Schoen-LectureDiffgeo} (note that now we need more explicit computations than \eqref{eq aysmptotic behavior}) to deduce that $d\mathcal{G}_x$ is a $L^p$ form for every $p < \frac{2n}{n-1}$ and $du$ is a $L^q$ form for every $q < \frac{2n}{1-\beta(2n-2)}$. Note that since we can pick $p < \frac{2n}{n-1}$ and $q < \frac{2n}{1-\beta (2n-2)}$ such that $\frac{1}{p} + \frac{1}{q} = 1$, by H\"older inequality, $d u \wedge d^c \mathcal{G}_x$ is a $L^1$ form. Approximating $u$ in $W^{1,q}$ by smooth functions $u_k$ (note that $0\leq u\leq 1$). By Green's formula for the smooth function $u_k$, we have
\begin{align*}u_k(x) &= \frac{1}{V_\omega} \int_X u_k(\omega +dd^c \mathcal{G}_x)\wedge \omega^{n-1} \\ &= \frac{1}{V_\omega} \int_X u_k \omega^n - \frac{1}{V_\omega} \int_X du_k \wedge d^c \mathcal{G}_x \wedge \omega^{n-1} \cdot
\end{align*}
By H\"older inequality, we can prove that
\[\frac{1}{V_\omega} \int_X du_k \wedge d^c \mathcal{G}_x \wedge \omega^{n-1} \rightarrow \frac{1}{V_\omega}\int_X du\wedge d^c \mathcal{G}_x \wedge \omega^{n-1} \text{ as } k \rightarrow \infty.\]
Therefore, let $k\rightarrow \infty$, we get
\[0 = u(x) = \frac{1}{V_\omega} \int_X u \omega^n - \frac{\beta}{V_\omega} \int_X \frac{d\mathcal{G}_x \wedge d^c \mathcal{G}_x \wedge \omega^{n-1}}{(-\mathcal{G}_x)^{1+\beta}}\cdot\]
Now, since $0\leq u \leq 1$, we infer that
\begin{equation}\label{eq estimate d Gx / Gx^1+beta}\frac{1}{V_\omega} \int_X \frac{|\nabla\mathcal{G}_x|_\omega^2 }{(-\mathcal{G}_x)^{1+\beta}} \omega^n \leq \frac{1}{\beta}\cdot\end{equation}

Let $s \in \Big(0,\frac{m-n}{2mn-m+n}\Big)$. Then we can pick $\beta = \beta(s) > 0$ such that $\frac{(1+s)(1+\beta)}{1-s} < \frac{mn}{mn-m+n}\cdot$
Now, by using H\"older inequality, \eqref{eq estimate d Gx / Gx^1+beta}, and part (ii), we have
\begin{align*}
\frac{1}{V_\omega}\int_X |\nabla G_x|^{1+s}_\omega \omega^n &= \frac{1}{V_\omega}\int_X |\nabla \mathcal{G}_x|^{1+s}_\omega \omega^n\\ &= \frac{1}{V_\omega} \int_X \frac{|\nabla \mathcal{G}_x|^{1+s}_\omega}{|\mathcal{G}_x|^{\frac{(1+s)(1+\beta)}{2}}} \cdot |\mathcal{G}_x|^{\frac{(1+s)(1+\beta)}{2}} \omega^n \\
&\leq \Big( \frac{1}{V_\omega} \int_X \frac{|\nabla \mathcal{G}_x|^2_\omega}{|\mathcal{G}_x|^{1+\beta}}\Big)^{\frac{1+s}{2}} \cdot \Big(\frac{1}{V_\omega}\int_X |\mathcal{G}_x|^{\frac{(1+s)(1+\beta)}{1-s}} \omega^n\Big)^{\frac{1-s}{2}} \\
&\leq \frac{1}{\beta} \cdot C_2^{\frac{1-s}{2}}\cdot
\end{align*}
We complete the proof for (iii).
\end{proof}
\begin{remark}
    We note that if we let $m\rightarrow \infty$, then $r$ can take values in $\Big(0,\frac{1}{n-1}\Big)$ and $s$ can take values in $\Big(0,\frac{1}{2n-1} \Big)$. Hence, if we use the exponential condition as in \cite[Theorem 1.1]{DiNezzaGG}, we can get the value ranges as in \cite{Guo-Phong-Song-Sturm,GuedjTo-diameter,GPSS_bodieukien}. 
\end{remark}

By the estimates for Green's functions in Theorem \ref{theorem uniform Green semipositive class}, we just need to follow arguments in  \cite{Guo-Phong-Song-Sturm2}   (see also \cite{GuedjTo-diameter}) to obtain the following Sobolev inequalities. 

\begin{theorem}\label{theorem uniform sobolev for semi-positive class}
    Let $1< r < \frac{mn}{mn-m+n}$ be a constant. Let $\omega \in \mathcal{V}_\mathcal{K}(X,n,m,A)$ be a K\"ahler metric. Then
    \begin{itemize}
        \item[(i)] There exists a constant $C_1 = C_1(n,m,A,r)$ such that 
        \[\Big(\frac{1}{V_\omega} \int_X |u - \overline{u}|^{2r} \omega^n \Big)^{\frac{1}{r}} \leq C_1 \Big(\frac{1}{V_\omega} \int_X du\wedge d^c u \wedge \omega^{n-1}\Big)\]
        for every $u\in W^{1,2}(X)$. Here, $\overline{u} = V_\omega^{-1} \int_X u \omega^n$.
        \item[(ii)] There exists a constant $C_2= C_2(n,m,A,r)$ such that
    \[\Big(\frac{1}{V_\omega} \int_X |u|^{2r} \omega^n \Big)^{\frac{1}{r}} \leq C_2 \Big( \frac{1}{V_\omega} \int_X du\wedge d^c u \wedge \omega^{n-1} + \frac{1}{V_\omega} \int_X |u|^2 \omega^n\Big)\]
        for every $u\in W^{1,2}(X)$.
    \end{itemize}    
\end{theorem}

\section{Sobolev inequalities for semi-positive classes}\label{section sobolev semi-positive}

Let $X$ be a compact K\"ahler manifold and $\omega_X$ be a fixed K\"ahler form on $X$. Let $T$ be a closed positive $(1,1)$-current such that $T = \omega + dd^c u$ where $\omega$ is a semi-positive form and $u$ is a bounded $\omega$-psh function. Assume that $T^n$ is a smooth measure.  Denote $\mu := V_T^{-1} T^n$. For $\rho \in [0,1]$, put $\omega_\rho = \omega + \rho \omega_X$. We define
\[\mu_\rho = V_{\omega_\rho}^{-1}\Big[T^n +  \Big(\frac{V_{\omega_\rho}-V_\omega}{V_{\omega_X}}\Big) \omega_X^n\Big]\cdot\]
Then $\mu_\rho$ is a smooth positive volume form. Consider $u_\rho$ be the unique solution of the equation
\begin{equation}\label{eq MA define u rho}
    V_{\omega_\rho}^{-1}(\omega_\rho + dd^c u_\rho)^n = \mu_\rho \text{ with } \sup_X u_\rho = 0.
\end{equation}
We put $T_\rho= T_{\rho,\omega} := \omega_\rho + dd^c u_\rho$. Thus, $T_\rho \rightarrow T$ as currents when $\rho \rightarrow 0^+$. Since $\omega_\rho$ is a K\"ahler form and $\mu_\rho$ is a smooth positive volume form, by Yau's theorem, $T_\rho$ is a K\"ahler metric. 

We denote by $\mathcal{V}_\text{semi}^*(X,n,m,A)$ be the set of closed positive $(1,1)$-currents $T$ such that the following two conditions hold:
\begin{itemize}
    
\item[(i)] $T = \omega + dd^c u$ where $\omega$ is a semi-positive big form and $u$ is a bounded $\omega$-psh function;

\item[(ii)] The probability measure $\mu = V_T^{-1} T^{n}$ is smooth such that 
    \begin{equation}\label{cond V*}
    \Big(\int_X (-\psi)^m d\mu\Big) \leq A
    \end{equation}
    for every $T$-psh function $\psi$ with $\sup_X \psi = 0$.
\end{itemize}
Recall that a dsh function is a difference of two quasi-psh function (see \cite{DS_tm}). A dsh function $\psi$ is said to be \textit{$T$-plurisubharmonic} ($T$-psh for short) if $T+ dd^c \psi \geq 0$ and $\sup_X \psi <+\infty$. Note that a dsh function is well-defined modulo a pluripolar set.

    We note that when $T \in \mathcal{V}_\text{semi}^*(X,n,m,A)$, we can have an upper bound for $u_\rho$.
\begin{lemma}\label{lemma upper bound u rho} Suppose that $T\in \mathcal{V}_\text{semi}^*(X,n,m,A)$. There exists an upper bound $M$ such that for every $\rho \in [0,1]$, we have
    \[\|u_\rho\|_{L^\infty} \leq M.\]
\end{lemma}
\begin{proof}
    Note that $\mu_\rho \leq \mu + V_{\omega_X}^{-1} \omega_X^n$. Let $\psi$ be a $\omega_\rho$-psh function such that $\sup_X \psi = 0$. Then $\psi$ is a $\omega_1$-psh function with $\sup_X \psi = 0$. Therefore, there exists a constant $B$ depending only on $\omega, \mu$ and $\omega_X$ such that
    \[\Big(\int_X (-\psi)^m d \mu_\rho \Big)^\frac{1}{m} \leq \Big(\int_X (-\psi)^m d (\mu + V_{\omega_X}^{-1}\omega_X^n) \Big)^\frac{1}{m} \leq B.\]
    By Theorem~\ref{theorem - uniform estimate Guedj-Lu condition}, we get the bound
    \[\|u_\rho\|_{L^\infty} \leq C B^{\frac{3}{n+2m-3}}=: M\]
    where $C= C(m,n)$.
\end{proof}

Since $\mathcal{V}^*_\text{semi}(X,n,m,A)$ is a natural generalization of $\mathcal{V}_\mathcal{K}(X,n,m,A)$, one may expect that Sobolev inequalities hold for $\mathcal{V}^*_\text{semi}(X,n,m,A)$. Unfortunately, we don't know if our method works for this class. Therefore, we consider the following class, which is still enough for our purpose. We denote by $\mathcal{V}_{\text{semi}}(X,n,m,A)$ be the set of closed positive $(1,1)$-currents $T$ of bounded potentials such that the following four conditions hold:
\begin{itemize}
\item[(i)] $T$ belongs to a semi-positive big cohomology class;

\item[(ii)] $T$ is a smooth K\"ahler form on an open dense Zariski set $U_T$ in $X$;

\item[(iii)] $T^n$ is a smooth measure;

\item[(iv)] There exist a semi-positive big form $\omega$ in the same cohomology class with $T$ and a sequence of positive numbers $(\rho_j)_j$  converging to $0$ so that 
$$\frac{1}{V_{T_{j}}} \int_X (-\psi)^m (T_{j})^n \le A,$$
for every $T_j$-psh function $\psi$ with $\sup_X \psi=0$,
where $T_{j} := T_{\rho_j}$. 
\end{itemize}

\begin{lemma}\label{lemma continuous V^* subset V}Let $T \in \mathcal{V}^*_\text{semi}(X,n,m,A)$. Assume that $u_\rho \rightarrow u$ in $\mathscr{C}^0$ topology as $\rho \rightarrow 0^+$ (thus $T$ has continuous potential). Then $T \in \mathcal{V}_\text{semi}(X,n,m,A+\varepsilon)$ for any $\varepsilon > 0$.
\end{lemma}

\proof For each $\rho$, since 
\[\Big\{\psi \ T_\rho\text{-psh}: \sup_X \psi = 0\Big\}\]
is a compact set in the $L^m$ topology, we can pick a $T_\rho$-psh function $\psi_\rho$ with $\sup_X \psi_\rho = 0$ such that
\[\frac{1}{V_{T_\rho}}\int_X (-\psi_\rho)^m (T_\rho)^n = A_\rho:= \sup \Big\{  \frac{1}{V_{T_\rho}}\int_X (-\psi)^m (T_\rho)^n : \psi \ T_\rho\text{-psh with }\sup_X \psi = 0\Big\}\cdot\]
 By Lemma~\ref{lemma upper bound u rho}, we have 
\[\Big\{ \psi_\rho + u_\rho : \rho \in (0,1]\Big\} \subset \Big\{\psi \ \omega_1\text{-psh such that } -M\leq \sup_X \psi \leq 0\Big\}\cdot\]
The latter set is a compact set in the $L^1$ topology. Thus, there exist a sequence $(\rho_j)_{j}$ and a $T$-psh function $\psi_0$ such that
\[\psi_{\rho_j}+ u_{\rho_j} \rightarrow \psi_0 + u \text{ as } j \rightarrow \infty.\] 
By Hartogs' lemma and the fact that $u_0$ is continuous, we have
\[\limsup\limits_{j\to \infty} \max_X (\psi_{\rho_j} + u_{\rho_j} - u_0) = \max_X \psi_0.\]
Since $u_{\rho_j}\rightarrow u_0$ in $\mathscr{C}^0$ topology as $j\rightarrow \infty$, we get
\[\limsup\limits_{j\to \infty} \max_X \psi_{\rho_j} = \max_X \psi_0.\]
Hence, $\sup_X \psi_0 = 0$. We infer that
\[\lim\limits_{j\rightarrow \infty} A_{\rho_j} = \frac{1}{V_T} \int_X (-\psi_0)^m T^n   \leq A.\]
The result follows.
\endproof

\begin{remark}
    There is a well-known conjecture (see \cite{Hiep_holder} and \cite[Question 19]{DinewGZ-openproblems}) that the potential of $T$ is continuous. So, we expect that $\mathcal{V}^*_\text{semi}(X,n,m,A) \subset \mathcal{V}_{\text{semi}} (X,n,m,A+\varepsilon)$ holds without the assumption in Lemma~\ref{lemma continuous V^* subset V} (this will simplify our proof). The conjecture was known to be true when the cohomology class of $T$ is integral  by \cite{Dinew_Zhang_stability,Coman-Guedj-Zeriahi}; see also \cite{Do-Vu-log-continuity} for a related work.
\end{remark}

Lemma~\ref{lemma continuous V^* subset V} shows immediately that $\mathcal{V}_\text{semi}$ is an extension of $\mathcal{V}_\mathcal{K}$.

\begin{corollary}\label{corollary V Kahler subset V semi} We have $\mathcal{V}_\mathcal{K}(X,n,m,A) \subset \mathcal{V}_{\text{semi}}(X,n,m,A+\varepsilon)$ for every $\varepsilon > 0$.
\end{corollary}

We now prove Sobolev inequalities for $T \in \mathcal{V}_{\text{semi}}(X,n,m,A)$. 
Recall that $T$ is a K\"ahler form  on an open Zariski dense subset $U_T$. Denote by $W^{1,2}_T(X)$ the set of Borel functions $u$ on $X$ such that $u, \nabla u$ are locally square integrable in $U_T$ and there hold
\[\int_{U_T} |u|^2 T^n < \infty \text{ and } \int_{U_T} du\wedge d^cu \wedge T^{n-1} < \infty.\]

Here is our main result of this section.

\begin{theorem}\label{theorem uniform sobolev for semi-positive class positive closed current}
    Let $1< r < \frac{mn}{mn-m+n}$ be a constant. Let $T$ be a closed positive current in $\mathcal{V}_\text{semi}(X,n,m,A)$. Then
    \begin{itemize}
        \item[(i)] There exists a constant $C_1 = C_1(n,m,A,r)$ such that 
        \[\Big(\frac{1}{V_T} \int_{U_T} |u - \overline{u}|^{2r} T^n \Big)^{\frac{1}{r}} \leq C_1 \Big(\frac{1}{V_T} \int_{U_T} du\wedge d^c u \wedge T ^{n-1}\Big)\]
        for every $u\in W^{1,2}_T(X)$. Here, $\overline{u} = V_T^{-1}\int_{U_T} u T^n$.
        \item[(ii)] There exists a constant $C_2= C_2(n,m,A,r)$ such that
    \[\Big(\frac{1}{V_T} \int_{U_T} |u|^{2r} T^n \Big)^{\frac{1}{r}} \leq C_2 \Big( \frac{1}{V_T} \int_{U_T} du\wedge d^c u \wedge T^{n-1} + \frac{1}{V_T} \int_{U_T} |u|^{2r} T^n\Big)\]
        for every $u\in W^{1,2}_T(X)$.
    \end{itemize}    
\end{theorem}

\begin{proof}
The proof is almost identical to that of \cite[Theorem 8.1]{Guo-Phong-Song-Sturm2} given that we have obtained uniform Sobolev inequalities for metrics in $\mathcal{V}_{\mathcal{K}}(X,n,m,A)$.

Let $T \in \mathcal{V}_\text{semi}(X,n,m,A)$ and $(T_j)_j$ be the sequence of smooth K\"ahler metrics as in the definition of $\mathcal{V}_\text{semi}(X,n,m,A)$.  If we have the property that $T_j$ converges smoothly locally to $T$, then the proof of \cite[Theorem 8.1]{Guo-Phong-Song-Sturm2} applies verbatim to our setting. There are three steps in the proof of \cite[Theorem 8.1]{Guo-Phong-Song-Sturm2}: first, (i) was proved for $u$ smooth compactly supported in $U_T$, and secondly, for $u$ bounded in $W^{1,2}_T(X)$, and thirdly for $u \in W^{1,2}_T(X)$ arbitrary. Only the first step requires the local smooth convergence of $T_j$ to $T$. It is sufficient to have that $(T_j)^n, (T_j)^{n-1}$ converge weakly to $T^n$, $T^{n-1}$ respectively. This weak convergence is known to be true because the sequence of potentials of $T_j$ almost decreases to that of $T$ (see \cite[Lemma 3.3]{Vu-diameter} for details). Hence the uniform Sobolev inequalities (i) and (ii) hold. 
\end{proof}

\section{Metrics on the space of singularity types}\label{section relative pluripotential theory}

Let $X$ be a compact K\"ahler manifold and $\omega_X$ be a fixed K\"ahler form on $X$. Let $\theta$ be a closed smooth $(1,1)$-form representing a big cohomology class in $X$. Recall that  $\theta$-psh functions $u$ and $v$ are said to  have the same singularity type if there exists some constant $C$ such that $u-C \leq v \leq u+C$. This is an equivalence relation and the equivalence class of $u$ (as a $\theta$-psh function) is denoted by $[u]=[u]_\theta$.

Following Demailly, we consider the upper envelope
\[V_\theta := \sup\{v \in \PSH(X,\theta) \text{ such that } v\leq 0\}.\]
A $\theta$-psh function $\phi$ is said to be \textit{minimal singularities} if $[\phi]=[V_\theta]$. We say that a closed positive current $T$ in a big cohomology class has \emph{minimal singularities} if $T=\theta+\ddc u$ for some smooth form $\theta$ cohomologous to $T$ and $\theta$-psh function $u$ with minimal singularity. One can see that this definition is independent of the choice of $\theta, u$.  

Let $\phi \in \PSH(X,\theta)$. The following envelope
\[P_\theta [\phi]:= \Big(\sup \{v \in \PSH(X,\theta) \text{ such that } v\leq 0, [v] = [\phi]\}\Big)^*\]
was introduced in \cite{Ross-WittNystrom} and developed further in \cite{Lu-Darvas-DiNezza-mono}. We say that $\phi$ is \textit{model ($\theta$-)potential} if $\phi = P_\theta[\phi]$.  A function  $\phi' \in \PSH(X,\theta)$ is said to have \textit{model type singularity} if there exists a model potential $\phi$ such that $[\phi']= [\phi]$. It is clear by definition that minimal singularity is of model type singularity. Another important example of potentials with model type singularity is $\theta$-psh functions with analytic singularities, see \cite{Lu-Darvas-DiNezza-mono}. We recall the following important observation.

\begin{lemma} (\cite[Lemma 2.2]{Lu-Darvas-DiNezza-logconcave}) \label{le-supenvelope} Let $u,\varphi$ be $\theta$-psh functions such that $\varphi$ is more singular than $P_\theta[u]$. Then we have 
$$\sup_X \varphi = \sup_X (\varphi- P_\theta[u]).$$
\end{lemma}

For $u,v\in \PSH(X,\theta)$, we put $\theta_u:= \theta+\ddc u$ and 
$$d_\theta([u]_\theta,[v]_\theta):= 2 \int_X \theta_{\max\{u,v\}}^n - \int_X \theta_u^n - \int_X \theta_v^n.$$
Observe that this definition is independent of the choice of $u,v$ in their singularity types (see \cite{Lu-Darvas-DiNezza-mono,WittNystrom-mono} or \cite{Viet-generalized-nonpluri}). One can think of $d_\theta([u]_\theta, [v]_\theta)$ as a sort of ``distance" on the space of singularity types. It was shown in \cite{Darvas-Lu-DiNezza-singularity-metric,Vu_DoHS-quantitative} that the function $d_\theta(\cdot,\cdot)$ is comparable to the pseudo-metric on the space of singularity types. But we don't need to use this property in the sequel. 
We also use $d_\theta(u,v)$ to denote $d_\theta([u]_\theta, [v]_\theta)$.  We recall the following result which tells us that in the values of $d_\theta(u,v)$ for different $\theta$ are somehow equivalent. 

\begin{proposition} \label{pro-dthetatuongduong} (\cite[Proposition 4.3]{Vu_DoHS-quantitative}) Let $A>0,\delta>0$ be  constants and let  $\theta, \theta'$ be real closed smooth $(1,1)$-forms with $\theta \le \theta' \le A \omega_X$. There exists a constant $C > 0$ independent of $\theta,\theta'$ such that 
$$d_\theta(u,v) \le d_{\theta'}(u,v) \le C \delta^{-1/n} \big(d_{\theta}(u,v)\big)^{1/n},$$
for $\theta$-psh functions $u,v$ with $\int_X \theta_u^n \ge \delta, \int_X \theta_v^n \ge \delta$.
\end{proposition}

The following crucial stability result is a direct consequence of \cite[Theorem 1.3]{Vu_DoHS-quantitative} (we refer also to \cite[Theorem 1.4]{Darvas-Lu-DiNezza-singularity-metric} for the case where $\theta_k= \theta_0$ for every $k$).

\begin{theorem}\label{theorem stability for MA with prescribed singularity}
   Let $p>1$, $A$, $\delta$ be positive constants. Let $(\theta_k)_{k\in \N}$ be a sequence of closed smooth $(1,1)$-forms converging to $\theta_0$ in $\mathscr{C}^0$ as $k\to \infty$ so that $\theta_k \le A \omega_X$ for every $k\ge 0$.  
 Let $f_k \geq 0$ be functions  such that $\|f_k\|_{L^p}$ are uniformly bound and $f_k \rightarrow f_0$ in $L^1$ topology as $k\to \infty$.
 Let $u_k \in \PSH(X,\theta_k)$ be such that $(\theta_k + dd^c u_k)^n = f_k \omega_X^n$ and $\sup_X u_k = 0$.
 Assume that $\int_X (\theta_k + dd^c u_k)^n \geq \delta$ for every $k$, and $d_{A\omega_X}([u_k], [u_0]) \to 0$ as $k \to \infty$.
    Then $u_k \rightarrow u_0$ in capacity (and therefore in $L^1$ topology) as $k\to \infty$.
\end{theorem}

Recall that a sequence of Borel functions $(u_k)_{k\in \N}$ is said to \textit{converge to a function $u$ in capacity} if for every $\varepsilon > 0$, $\capa_{\omega_X} (\{|u_k - u|\geq \varphi \})$ converges to $0$ as $k\rightarrow \infty$. Here, $\capa_{\omega_X}$ is the Monge-Amp\`ere capacity with respect to $\omega_X$, and defined by
\[\capa_{\omega_X}(E):= \sup_{\{\psi \in \PSH(X,\omega_X), -1\leq \psi \leq 0\}} \frac{1}{V_{\omega_X}}\int_E (\omega_X + dd^c \psi)^n\]
for every Borel subset $E$ of $X$.

In our application later, we will apply Theorem \ref{theorem stability for MA with prescribed singularity} to a specific situation in which the $L^1$ convergence of $u_k$ to $u_0$ is indeed well-known. The crucial point is the convergence in capacity of $u_k$ to $u_0$. The role of this property in our method can be seen through the following result which is similar to  \cite[Theorem 4.9]{Viet-generalized-nonpluri} and \cite[Theorem 2.3]{Lu-Darvas-DiNezza-mono}.

\begin{theorem}\label{theorem convergence in capacity}
    Let $X$ be a compact K\"ahler manifold of dimension $n$. Let $T_{jk}$ be closed positive $(1,1)$-currents for $1\leq j \leq m \leq n$, $k\in \N$ such that $T_{jk}$ converges to $T_j$ as currents when $k\rightarrow \infty$. Let $(U_\alpha)_\alpha$ be a finite covering of $X$ by open subsets such that $T_{jk} = dd^c u_{jk,\alpha}$ and $T_{j} = dd^c u_{j,\alpha}$ on $U_\alpha$ for every $\alpha$. Assume that
    \begin{itemize}
        \item[(i)] for every $\alpha$ and $j$, $u_{jk,\alpha} \rightarrow u_{j,\alpha}$ in capacity as $k\rightarrow \infty$;
        \item[(ii)] $\left \langle \{T_{1k}\} \wedge \cdots \wedge \{T_{mk}\} \right \rangle \rightarrow \left \langle \{T_{1}\} \wedge \cdots \wedge \{T_{m}\} \right \rangle$;
        \item[(iii)] $T_1,\ldots,T_m$ are of minimal singularities.
    \end{itemize} Then
    \[\left \langle T_{1k} \wedge \cdots \wedge T_{mk} \right \rangle \rightarrow \left \langle T_{1} \wedge \cdots \wedge T_{m} \right \rangle \text{ as currents.}\]
\end{theorem}

Here, the product of cohomology classes in (ii)  was introduced in \cite[Definition 1.17]{BEGZ}. We note that the capacity here is the local Monge-Amp\`ere capacity of Bedford-Taylor \cite{Bedford_Taylor_82} and since $\omega_X$ is a K\"ahler form, the global capacity and local capacity are equivalent. The proof of this theorem is similar to \cite[Theorem 2.3]{Lu-Darvas-DiNezza-mono}, we present here for reader's convenience. We first recall the following classical lemma.

\begin{lemma}\label{lemma local converge in capacity} (\cite[Theorem 4.26]{GZbook} or \cite[Theorem 1.11]{Kolodziej05})
    Let $\Omega \subset \mathbb{C}^n$ is an open subset. Suppose that $(f_k)_k$ are uniformly bounded quasi-continuous functions which converge in capacity to another quasi-continuous function $f$. Let $(u_k^1),\ldots,(u_k^m)$ be uniformly bounded plurisubharmonic functions on $\Omega$ converging in capacity to $u^1,\ldots,u^m$ respectively for $1\leq m \leq n$. Then
    \[f_k dd^c u_{k}^1 \wedge \cdots \wedge dd^c u_k^m \rightarrow f dd^c u^1 \wedge\cdots \wedge dd^c u^m \text{ as currents.}\]
\end{lemma}

\begin{proof}[Proof of Theorem~\ref{theorem convergence in capacity}]
    Let $\Phi$ be a weakly positive $(n-m,n-m)$-form on $X$. Using a partition of unity for $(U_\alpha)_\alpha$, we can write $\Phi = \sum_\alpha \Phi_\alpha$ where $\Phi_\alpha$ is a weakly positive form on $X$ has support on $U_\alpha$. For $C > 0$, let $u_{jk,\alpha}^C := \max (u_{jk,\alpha},-C)$ and $u_{j,\alpha}^C = \max (u_{j,\alpha},-C)$. Define also $f_{k,\alpha}^C := \mathbf{1}_{\cap_{j=1}^m\{u_{jk,\alpha} > -C\}}$ and $f_{\alpha}^C := \mathbf{1}_{\cap_{j=1}^m\{u_{j,\alpha} > -C\}}$. Now, by the definition of non-pluripolar product and Lemma~\ref{lemma local converge in capacity}, we infer that
    \[\liminf_{k\rightarrow \infty} \int_X \left \langle T_{1k} \wedge \cdots \wedge T_{mk} \right \rangle \cdot \Phi _\alpha \geq \int_X \left \langle T_{1} \wedge \cdots \wedge T_{m} \right \rangle \cdot \Phi _\alpha.\]
    Taking the sum over $\alpha$, we get
    \[\liminf_{k\rightarrow \infty} \int_X \left \langle T_{1k} \wedge \cdots \wedge T_{mk} \right \rangle \cdot \Phi  \geq \int_X \left \langle T_{1} \wedge \cdots \wedge T_{m} \right \rangle \cdot \Phi .\]

    By condition (ii), we can bound uniformly the mass of $\left \langle T_{1k} \wedge \cdots \wedge T_{mk} \right \rangle$. Thus, we can consider a limit current $S$ of this sequence, and because of above argument, $S \geq \left \langle T_{1} \wedge \cdots \wedge T_{m} \right \rangle $. Condition (iii) implies that $\{S\} \leq \{\left \langle T_{1} \wedge \cdots \wedge T_{m} \right \rangle\}$ (see \cite[Theorem 4.4]{Viet-generalized-nonpluri}). Thus, $S$ must be equal to $\left \langle T_{1} \wedge \cdots \wedge T_{m} \right \rangle$. We complete the proof.
\end{proof}

Let $S = \theta + dd^c u$ be a closed positive $(1,1)$-current in the cohomology class $\{\theta\}$. We said $S$ is of \textit{model type singularity} if its potential $u$ is of model type singularity as a $\theta$-psh function. Given a positive closed current $S$ of model type singularity, we can define
\[\kappa_\theta(S):= \| u - P_\theta[u]\|_{L^\infty},\]
where $u\in \PSH(X,\theta)$ such that $S = \theta + dd^c u$ and $\sup_X u = 0$.

Recall that a dsh function $\psi$ is said to be $S$-psh if $S+ dd^c \psi \geq 0$ and $\sup_X \psi <+\infty$. If $\psi$ is $S$-psh for some current $S$ of model type singularity, we said $\psi$ is a \textit{good dsh function}. 
We now prove an integrability property for these functions.

\begin{lemma}\label{modeltypeSkoda}
    Let $A,B$ be constants. Let $\theta$ be a closed $(1,1)$-form with $\theta \le B \omega_X$. For every current $S$ in $\{\theta\}$ of model type singularity such that $\kappa_\theta(S)\leq A$ and for every $S$-psh function $\psi$ with $\sup_X \psi = 0$, there exist positive constants $A_1,A_2$ depending only on $\omega_X, B$ such that 
    \[\int_X e^{-A_1 \psi} \omega_X^n \leq  A_2 e^{A_1 A}.\]
\end{lemma}
\begin{proof}  Write $S = \theta + dd^c u$ with $\sup_X u = 0$. Since $\psi$ is $S$-psh, $\varphi:= \psi+u$ is a $\theta$-psh function. Observe $\varphi \le u \le P_\theta[u] \le 0$ because $\sup_X \psi =0$.  Using $\kappa_\theta (T) \leq A$, we have
\[-A + P_\theta[u] \leq u \leq A + P_\theta[u].\]
This coupled with  Lemma \ref{le-supenvelope} gives
$$\sup \varphi = \sup_X (\varphi - P_\theta[u]) \ge \sup_X (\varphi -u) + \inf_X(u- P_\theta[u]) \ge -A. $$
By this and integrability property of $\omega_X$-psh functions, there exist positive constants $A_1,A_2$ depending only on $\omega_X, A$ such that 
    \[\int_X e^{-A_1 \psi} \omega_X^n  \leq \int_X e^{-A_1 \varphi} \omega_X^n \leq e^{A_1 A}\int_X e^{-A_1 (\varphi-\sup_X \varphi)} \omega_X^n \leq A_2 e^{A_1 A}.\]
 The proof is completed.
\end{proof}

\section{\texorpdfstring{$L^\infty$}--estimate for measures dominated by capacity}

The goal of this section is to present a $L^\infty$-estimate for solutions of complex Monge-Amp\`ere equations in the prescribed singularity setting. This was more or less settled in \cite{Lu-Darvas-DiNezza-logconcave,Lu-Darvas-DiNezza-logconcave,DiNezzaGG} developing further \cite{BEGZ,Kolodziej_Acta}, see also \cite{Vu_Do-MA}. The main point is to obtain explicit estimates which are crucial for our application later. We underline that we will not only need to use this $L^\infty$-estimate for measures with density of uniform $L^p$-bound but also for measures satisfying only Ko{\l}odziej's capacity conditions. 

Let $\theta$ be a smooth closed form in a big cohomology class and let $\phi$ be a  $\theta$-psh function. Recall that the (normalized) Monge-Amp\`ere capacity $\capa_\phi$ with respect to the potential $\phi$, defined by
\[\capa_{\phi}(E):= \sup_{\{\psi \in \PSH(X,\theta), \phi-1\leq \psi \leq \phi\}} \frac{1}{\int_X \theta_\phi^n} \int_E \theta_\psi^n\]
for every Borel subset $E$ of $X$. This generalizes the notion of $\capa_{\omega_X}$ in \cite{Bedford_Taylor_82,Kolodziej08holder}.  We refer to \cite{DiNezzaLu-capacity,Lu-Darvas-DiNezza-mono} for more information about this capacity.

\begin{theorem}\label{theorem Linfty for MA with pres sing}
    Let $\theta$ be a closed $(1,1)$-form in a big cohomology class.  Let $\phi$ be a $\theta$-psh function such that $\phi$ is a model potential. Let $\mu$ be a probability measure on $X$. Let $A,B,\delta$ be positive constants. Suppose that $\mu$ satisfies the following conditions:
    \begin{itemize}
        \item[(i)] $\mu(E) \leq A [\capa_\phi(E)]^{1+\delta}$ for every Borel subset $E $ of $X$;
        \item[(ii)] $\int_X (-u) d\mu \leq B$ for every $\theta$-psh function $u$ with $\sup_X u = 0$.
    \end{itemize}
    Let $u$ be the unique solution of the equation $\theta_u^n = (\int_X \theta_u^n) \cdot \mu$ with $\sup_X u = 0$. Then $[u] = [\phi]$ and
   $$ \kappa_\theta(\theta + dd^c u) = \|u-\phi\|_{L^\infty} \leq 
 2^{1+n/\delta}A^{1/\delta}B+1+ (1-2^{-\delta})^{-1}.$$
\end{theorem}

We refer to Proposition \ref{pro-Lmcondition} for examples of measures satisfying the condition in the above theorem. 

\begin{proof}For $t>0$, we put $g(t):= [\capa_{\phi}(u<\phi - t)]^{\frac{1}{n}}$. By comparison principle (note that $P_\theta[u] = \phi$), arguing as in  the proof of \cite[Theorem 3.3]{Lu-Darvas-DiNezza-logconcave}, we get
\begin{align}\label{ine-comparcapa}
s^n (g(t+s))^n \leq \int_{\{u< \phi-t\}}\theta_u^n = \mu(\{u< \phi - t\}) \leq A (g(t))^{(1+\delta)n}.
\end{align}
Hence $sg(t+s) \leq A^{\frac{1}{n}} (g(t))^{1+\delta}$ for every $t>0$ and $s\in [0,1]$. Let $t_0 > 0$ be a constant such that $g(t_0) \le  (2^nA)^{-\frac{1}{\delta n}}$. It was proved in \cite[Lemma 2.4]{EGZ} that $g(t) = 0$ for every $t \geq t_0 + \frac{1}{1-2^{-\delta}}$. Applying (\ref{ine-comparcapa}) to $s=1$ gives
\[(g(t+1))^n \leq \mu(\{u<\phi - t\}) \leq \int_X \frac{|\phi - u|}{t} d \mu \leq \frac{2B}{t}.\]
Choosing $t_0: = 2^{1+n/\delta}A^{1/\delta}B+1$, we get $g(t_0) \leq (2^nA)^{-\frac{1}{\delta n}}$. It follows that $g(t) = 0 $ for every 
$$t >t_1:=2^{1+n/\delta}A^{1/\delta}B+1+ (1-2^{-\delta})^{-1}.$$
Therefore, since $u\leq \phi$, we get $\|u-\phi\|_{L^\infty} \leq t_1$ as desired.
\end{proof}

The following result tells us that model potentials are preserved under modifications.

\begin{proposition}\label{prop model potential through blow up}
    Let $\pi: Y \to X$ be a surjective holomorphic map between compact K\"ahler manifolds. Assume that $\pi$ is biholomorphic outside a proper analytic subset of $Y$. Let $\theta$ be a closed smooth $(1,1)$-form on $X$. Let $\phi$ be a model $\theta$-potential. Then $\phi \circ \pi$ is a model $\pi^* \theta$-potential and we have 
    \begin{equation}\label{eq compare capa of blowup} \capa_\phi (E) = \capa_{\phi \circ \pi} (\pi^{-1}(E))\end{equation}
    for every Borel subset $E$ of $X$.
\end{proposition}

\begin{proof} 
   We need to show that $\phi \circ \pi = P_{\pi^* \theta}[\phi]$. It is clear that $\phi \circ \pi \leq P_{\pi^* \theta}[\phi\circ \pi]$. Let $\psi$ be a $\pi^* \theta $-psh function such that $\psi \leq 0$ and $[\psi] = [\phi\circ \pi]$. It follows that  $\tau = \pi_* \psi$ is  a $\theta$-psh function on $X$ with $\tau \leq \phi + O(1)$. Since $\phi$ is model potential, we must have $\tau \leq \phi $. Thus, $\psi \leq \phi \circ \pi$ and we infer that $\phi \circ \pi$ is a model potential.

    Let $\psi$ be a $\pi^* \theta$-psh function such that $\phi \circ \pi -1 \leq \psi \leq \phi \circ \pi$. Then, we have $\phi-1 \leq \tau \leq \phi$. Let $E$ be a Borel subset of $X$. We have
    \[\int_E (\theta + dd^c \tau)^n =  \int_{\pi^{-1}(E)} (\pi^* \theta + dd^c \psi)^n.\]
    We infer that
    \[ \capa_\phi (E) \geq \capa_{\phi \circ \pi} (\pi^{-1}(E)).\]
    Arguing similarly for the opposite direction gives \eqref{eq compare capa of blowup}.
\end{proof}

We finish this section by recalling a known fact that Ko\l odziej's capacity condition is weaker than $L^m$ integrable condition which was considered in Section~\ref{section uniform for Laplace}.

\begin{proposition} \label{pro-Lmcondition}
Let $X$ be a compact K\"ahler manifold of dimension $n$ and $m>n$ be a constant. Let $\theta$ be a closed $(1,1)$-form in a big cohomology class and $\mu$ be a probability measure on $X$. Let $\phi$ be a $\theta$-psh function. Suppose that $\phi$ is a model potential and $\int_X \theta_\phi^n > 0$. Assume that for every $\theta$-psh function $u$ with $\sup_X u = 0$, there exists a constant $A$ independent of $u$ such that
\[
\int_X (-u)^m d\mu \leq A.
\]
 Then there exists a positive constant $B = B(A,m,n)$ such that 
$$\mu(E) \leq B \big(\capa_{\phi}(E)\big)^{m/n},$$
for every Borel subset $E$ in $X$.
\end{proposition}

\begin{proof} The proof is standard, see, e.g., \cite[Proposition 2.4]{DinhVietanhMongeampere}. For every Borel subset $E$ of $X$,  we define
    \[M_\phi(E):= \sup_X \Big( \sup \{\psi \in \PSH(X,\theta) \text{ s.t. } \psi \leq \phi + O(1) \text{ on }X \text{ and } \psi \leq \phi \text{ on } E\} \Big)^*.\]
    Recall that, by \cite[Lemma 3.9]{Lu-Darvas-DiNezza-logconcave} (see also \cite[Lemma 4.9]{Lu-Darvas-DiNezza-mono} and \cite[Proposition 6.1]{GZ}), we have
    \begin{equation}\label{eq compare cap}
    1 \leq \Big(\capa_\phi(K)\Big)^{-\frac{1}{n}} \leq \max (1,M_\phi(K))\end{equation}
    for every compact subset $K$ of $X$. It is enough to obtain our inequality for small compact set $K$. We can also assume that $K$ is non-pluripolar, otherwise both side are zero. 

    If $M_\phi(K) \leq 2$, then $\capa_\phi(K)  \ge  \frac{1}{2}$, and we can pick $B = 2^{-1-\delta}$. Assume now that $M_\phi(K) \ge  2$. Let $\psi \in \PSH(X,\theta)$ such that $\psi \leq \phi + O(1), \psi \leq \phi$ on $K$, and $\sup_X \psi \geq M_\phi(K) - 1$. Put $u = \psi - \sup_X \psi$. Thus, $u$ is a $\theta$-psh function with $\sup_X u = 0$. Since $\psi \leq \phi$ on $K$, we have
    \[K \subset \{u \leq \phi + 1 - M_\phi(K)\} \subset \{u \leq 1 - M_\phi(K)\} \subset \{(-u)^m \geq (M_\phi(K)-1)^{m}\}.\]
    Thus, by \eqref{eq compare cap}, we have
    \begin{align*}
\mu(K) &\leq \int_{\{(-u)^m \geq (M_\phi(K))^{m}\}} d\mu \leq \int_X \frac{(-u)^m}{(M_\phi(K))^{m}} d\mu \\ &\leq \frac{A}{(M_\phi(K)-1)^{m}} \leq \frac{2^m A}{M_\phi(K)^{m}} \le 2^m A \capa_\phi^{m/n}(K).
\end{align*}
This finishes the proof.
\end{proof}

\section{Sobolev inequalities for big cohomology classes}

We now prove Sobolev inequalities for metrics in $\mathcal{W}_{\text{big}}(X,n,p,A,B)$. We begin with some auxiliary lemmas. 

\begin{lemma}\label{lemma iii implies 1}
    Let $T \in \mathcal{W}_{\text{big}}(X,n,p,A,B)$. For every $m\geq 1$, there exists a constant $C = C(\omega_X,n,p,B,m)$ such that
    \[\frac{1}{V_T} \int_X (-\psi)^m T^n \leq C\]
    for every $\omega_X$-psh function $\psi$ with $\sup_X \psi = 0$.
\end{lemma}
\begin{proof} Let $q:=p/(p-1)$.
    By H\"older's inequality, we have
    \begin{align*}
    \frac{1}{V_T}\int_X (-\psi)^m T^n =\int_X (-\psi)^m f_T \omega_X^n &\leq \|f_T\|_{L^p}\cdot \Big(\int_X (-\psi)^{qm}\omega_X^n\Big)^{\frac{1}{q}} 
    \end{align*}
    which is bounded by a constant depending only on $p,A,B,n,\omega_X$ because of exponential integrability of $\omega_X$-psh functions. 
\end{proof}

The following more general fact is useful for us. We refer to \cite[Theorem 1.8]{GuedjTo-diameter} for the case of minimal singularities.

\begin{lemma}\label{lemma iii implies 2}
    Let $T \in \mathcal{W}_{\text{big}}(X,n,p,A,B)$. Let $S$ is a closed positive $(1,1)$-current of model type singularity such that
    \begin{itemize}
        \item[(i)] There exists a closed $(1,1)$-form $\theta$ in the same cohomology class with $S$ such that $\theta \leq A_1\omega_X$;
        \item[(ii)] $\kappa_\theta(S) \leq A_2$.
    \end{itemize}
    Then for every $m\geq 1$, there exists a constant $A_3$ depending only on $A_1,\omega_X,n,p,$ $A,B,m$ such that
    \[\frac{1}{V_T} \int_X (-\psi)^m T^n \leq A_3+ 2^m A_2^m\]
    for every $S$-psh function $\psi$ with $\sup_X \psi = 0$.
\end{lemma}
\begin{proof} Write $S = \theta + dd^c u$ for some $\theta$-psh function $u$ with $\sup_X u = 0$. Put $\varphi = \psi+u$. As in the proof of Lemma~\ref{modeltypeSkoda}, the condition $\kappa_\theta(S) \leq A_2$ implies that $-A_2\leq \sup_X \varphi \leq A_2$. The result thus follows from Lemma~\ref{lemma iii implies 1}
\end{proof}

We now proceed with the proof of Theorem \ref{theorem Sobolev for big case}. Let $T\in \mathcal{W}_{\text{big}}(X,n,p,A,B)$. 
In the first step, we will approximate $T$ by a sequence of metrics with analytic singularity.  By hypothesis, there exist a constant $A'\ge 1$ independent of $T$ and a closed smooth form $\theta$ cohomologous to $T$ such that $\theta \le A' \omega_X$. Let $\varphi$ be the $\theta$-psh function such that $T=\theta + dd^c \varphi$ and $\sup_X \varphi = 0$. By Demailly's analytic regularization theorem (\cite{Demailly_analyticmethod}), there exist closed positive currents $(T_k)_{k\in \N}$ and quasi-psh functions $(\varphi_k)$ such that:
\begin{itemize}
    \item[(i)] $T_k = \theta_k + dd^c \varphi_k$ where $\theta_k = \theta + \omega_X / k$; 
    \item[(ii)] $\varphi_k$ has analytic singularity and $\sup_X \varphi_k$ is bounded from above uniformly; 
    \item[(iii)] $\varphi\leq \varphi_k$ and $\varphi_k$ converges pointwise and in $L^1$ topology to $\varphi$ as $k\to \infty$;
    \item[(iv)] $T_k$ is smooth on $U_T$ (because $T$ is so).
\end{itemize}

Let $(f_k)_{k\in \N}$ be a sequence of smooth positive functions such that $\|f_k - f_T\|_{L^p(\omega_X^n)}\leq\frac{1}{k}$. We define 
$$\mu_k :=c_k f_k \omega_X^n,$$
 where $c_k$ is the normalized constant so that $\mu_k$ is a probability measure. Let $\widetilde{\varphi}_k$ be the solution of the Monge-Amp\`ere equation with prescribed singularity:
\[ (\theta_k + dd^c \widetilde{\varphi}_k)^n = V_{T_k} \mu_k \text{ with } [\widetilde{\varphi}_k] = [\varphi_k] \text{ and } \sup_X \widetilde{\varphi}_k = 0.\]
This is always solvable since $[\varphi_k]$ is a model type singularity. 
Put 
$$\widetilde{T}_k:= \theta_k+ \ddc \widetilde{\varphi}_k,\quad u_k:= \widetilde{\varphi}_k- \varphi_k$$
 which is a bounded dsh function (but we don't know if we can control uniformly $L^\infty$-norm of $u_k$). 

\begin{lemma} \label{le-step1xapxikidigitich} The sequence $\widetilde{\varphi}_k$ converges in capacity to $\varphi$ as $k\to \infty$. Moreover, there exists a constant $A_1 = A_1(\omega_X,n,p,A,B)$ such that 
\begin{equation}\label{eq control kappa Tk nga}\kappa_{\theta_k}(\widetilde{T}_k) = \kappa_{\theta_k}(\theta_k + dd^c \widetilde{\varphi}_k) \leq A_1.
\end{equation}    
\end{lemma}

\proof
Since $\varphi_k \geq \varphi$,  using Theorem \ref{theorem convergence in capacity}, and the fact $T$ has minimal singularity, we get 
$$(\theta_k+ \ddc \varphi_k)^n \to (\theta+ \ddc \varphi)^n \text{ as } k\to \infty.$$
 Hence, $d_{\theta_k}([\varphi_k],[\varphi]) \rightarrow 0$ as $k\to \infty$. This combined with Proposition \ref{pro-dthetatuongduong} gives that $d_{2A' \omega_X}([\varphi_k],[\varphi]) \rightarrow 0$ as $k \to \infty$. 
Since $[\widetilde{\varphi}_k] = [\varphi_k]$, we deduce that
$d_{2A' \omega_X}([\widetilde{\varphi}_k],[\varphi]) \rightarrow 0$ as $k \to \infty$. 
By Theorem~\ref{theorem stability for MA with prescribed singularity}, we see that 
$\widetilde{\varphi}_k$ converges in capacity to $\varphi$ as $k\to \infty$.

Since $\theta_k \leq 2A' \omega_X$ and $\|f_k\|_{L^p(\omega_X^n)} \leq B + 1$, by Theorem~\ref{theorem Linfty for MA with pres sing} and Lemma \ref{lemma iii implies 1}, the second desired assertion follows.
\endproof

 Let $\pi_k: X_k \rightarrow X$ be a smooth modification so that there are an effective $\R$-divisor $D_k$ and a semi-positive closed form $\omega_k$ on $X_k$ satisfying
\[(\pi_k)^* T_k = \omega_k + [D_k].\]
Since $[\widetilde{\varphi}_k] = [\varphi_k]$, we can pick bounded dsh function $u_k$ such that $\widetilde{T}_k = T_k + dd^c u_k$. Hence, we can write $(\pi_k)^*\widetilde{T}_k =  \omega_k + dd^c (u_k \circ \pi_k) + [D_k]$. Put $\widehat{T}_k = \omega_k + dd^c (u_k \circ \pi_k)$. We observe that: 
\begin{itemize}
    \item[(i)] $\widehat{T}_k$ is a closed positive $(1,1)$-current of bounded potentials and belongs to a semi-positive class;
    \item[(ii)] $(\widehat{T}_k)^n = \pi_k^* ((\widetilde{T}_k)^n )= \pi_k^* (V_{T_k} \mu_k)$ is a smooth volume form;
\end{itemize}

Let $r\in \Big(1,\frac{n}{n-1}\Big)$ be a constant. Fix a constant $m$ big enough such that $m$ satisfies condition~\eqref{eq cond for m} and 
$$r < \frac{mn}{mn-m + n}\cdot$$

\begin{lemma} \label{le-step2-giaikidichankapp}
There exists a constant $C$ independent of $T$ and $k$ such that 
\[\widehat{T}_k \in \mathcal{V}_{\text{semi}}(X_k,n,m,C).\]
\end{lemma}

\proof
Let $\omega_{X_k}$ be a K\"ahler form on $X_k$. As in Section~\ref{section sobolev semi-positive}, for $\rho \in (0,1]$, we define $\widehat{T}_{k,\rho}:= \omega_k + \rho \omega_{X_k} + dd^c u_{\rho,k}$ where $u_{\rho,k}$ is the unique solution of the equation
\[(\widehat{T}_{k,\rho})^n = (\omega_k + \rho \omega_{X_k} + dd^c u_{\rho,k})^n = (\widehat{T}_k)^n + c_\rho \omega_{X_k}^n \text{ with } \sup_{X_k} u_{\rho,k} = 0,\]
where $c_\rho$ is a normalized constant so that the masses on both sides are equal. Observe that $c_\rho \to 0 $ as $\rho \rightarrow 0^+$.  Put 
$$\widetilde{T}_{k,\rho}:= (\pi_k)_* \widehat{T}_{k,\rho}.$$
Observe 
\begin{align}\label{eq-classTngarhok}
\{\widetilde{T}_{k,\rho}\} =\{\widetilde{T}_k\} + \rho \{(\pi_k)_* (\omega_{X_k})\}=\{\theta_k\}+\rho \{(\pi_k)_* (\omega_{X_k})\}.
\end{align}
We see that there exist a positive constant $A_k$ (depending on $\omega_{X_k}$) and a smooth form $\theta_{k,\rho} \in \{\widetilde{T}_{k,\rho}\}$ such that 
\begin{equation}\label{eq bound cohom of theta k rho}\theta_{k,\rho} \leq (2A'+ \rho A_k) \omega_X.\end{equation}

Now we want to control $\kappa_{\theta_{k,\rho}}$ of $\widetilde{T}_{k,\rho}$. Recall that $\widetilde{T}_{k,\rho}$ satisfies
\[V_{\widetilde{T}_{k,\rho}}^{-1}(\widetilde{T}_{k,\rho})^n = V_{\widetilde{T}_{k,\rho}}^{-1}(\pi_k)_* \Big((\widehat{T}_k)^n + c_\rho \omega_{X_k}^n\Big) = V_{\widetilde{T}_{k,\rho}}^{-1} V_{T_k}\mu_k + c_\rho V_{\widetilde{T}_{k,\rho}}^{-1}(\pi_k)_* (\omega_{X_k}^n).\]
The idea is to use Theorem~\ref{theorem Linfty for MA with pres sing} to bound $\kappa_{\theta_{k,\rho}}(\widetilde{T}_{k,\rho})$. By Proposition~\ref{prop model potential through blow up}, we see that, for a model $\theta_{k,\rho}$-potential $\phi$,
\[\capa_\phi(E) = \capa_{\phi \circ \pi_k}(\pi_k^{-1}(E))\]
for every Borel subset $E$ of $X$. Write $\widetilde{T}_{k,\rho} = \theta_{k,\rho} + dd^c \widetilde{\varphi}_{k,\rho}$ with $\sup_X \widetilde{\varphi}_{k,\rho} = 0$ and put $\phi = P_{\theta_{k,\rho}}[\widetilde{\varphi}_{k,\rho}]$. By \cite[Proposition 3.10]{Lu-Darvas-DiNezza-logconcave} (note that $\phi \circ \pi_k$ is a model potential by Proposition~\ref{prop model potential through blow up}), for every Borel subset $E$ of $X$, we can bound
\[\int_E (\pi_k) _* (\omega_{X_k}^n) = \int_{\pi_k^{-1}(E)}\omega_{X_k}^n \leq F_k [\capa_{\phi}(E)]^2,\]
where $F_k$ is a constant which does not depend on $\rho$. By Lemma \ref{lemma iii implies 1} and Proposition \ref{pro-Lmcondition}, we can bound
\[\mu_k \leq F [\capa_\phi(E)]^2\]
for some constant $F= F(\omega_X,n,p,A,B)$. Thus, we get
\[V_{\widetilde{T}_{k,\rho}}^{-1} \int_E (\widetilde{T}_{k,\rho})^n \leq (F+ c_\rho F_k) [\capa_\phi(E)]^2\]
for every Borel subset $E$ of $X$.

Let $\psi$ be a $\theta_{k,\rho}$-psh function with $\sup_X \psi = 0$. By \eqref{eq bound cohom of theta k rho}, $\psi$ is a $(2A'+\rho A_k)\omega_X$-psh function. Therefore,
\[\int_X (-\psi) V_{T_k} \mu_k \leq G + \rho G_k\]
where $G = G(\omega_X,n,p,A,B)$ and $G_k$ are constants independent of $\rho$. Moreover, since $\psi$ is a $(2A'+A_k)\omega_X$-psh function, we obtain
\[\int_X (-\psi) (\pi_k)_* (\omega_{X_k}^n) = \int_{X_k} (-\psi \circ \pi_k) \omega_{X_k}^n \leq H_k\]
for some constant $H_k$ independent of  $\rho$. 
Thus, by Theorem~\ref{theorem Linfty for MA with pres sing}, we get the crucial estimate
\begin{equation}\label{bound kappa of theta}
\kappa_{\theta_{k,\rho}}(\widetilde{T}_{k,\rho}) \leq M + c_\rho M_k,
\end{equation}
for some positive constants $M$ depending only on $\omega_X,n,p,A,B$ and  $M_k$ independent of  $\rho$.  Let $\widehat{\psi}_\rho$ be the $\widehat{T}_{k,\rho}$-psh function such that $\sup_{X_k} \psi_{\rho} = 0 $.
Put $\widetilde{\psi}_\rho = (\pi_k)_* \widehat{\psi}_\rho$. Thus $\widetilde{\psi}_\rho$ is a $\widetilde{T}_{k,\rho}$-psh function with $\sup_X \widetilde{\psi}_\rho = 0$.
By (\ref{bound kappa of theta}) and Lemma \ref{lemma iii implies 2}, we see that there exists a constant $C_1>0$ independent of $k,\rho,T$ such that 
$$\sup_{\widehat{\psi}_\rho}\frac{1}{V_{\widehat{T}_{k,\rho}}} \int_{X_k} (-\widehat{\psi}_\rho)^m (\widehat{T}_{k,\rho})^n =\sup_{\widetilde{\psi}_\rho}\frac{1}{V_{\widetilde{T}_{k,\rho}}} \int_{X} (-\widetilde{\psi}_\rho)^m (\widetilde{T}_{k,\rho})^n \le C_1  + o(1) \text{ as } \rho \rightarrow 0^+.$$
 Hence, $\widehat{T}_k \in \mathcal{V}_{\text{semi}}(X_k,n,m,C)$ for some constant $C>0$ independent of $k,T$.   
\endproof

\begin{proof}[End of the proof of Theorem \ref{theorem Sobolev for big case}]
Since $r < \frac{mn}{mn-m + n}$, by Theorem~\ref{theorem uniform sobolev for semi-positive class positive closed current}  and Lemma \ref{le-step2-giaikidichankapp}, we get the Sobolev inequalities for $\widehat{T}_k$ (thus for $\widetilde{T}_k$ because $\pi_k$ is biholomorphic outside some proper analytic subset of $X_k$).

By Lemma \ref{le-step1xapxikidigitich} and  Theorem~\ref{theorem convergence in capacity}, we get $\widetilde{T}_k^l \to T^l$ weakly as $k\to \infty$ for every $1\le l \le n$. This combined with approximation arguments as in the proof of Theorem \ref{theorem uniform sobolev for semi-positive class positive closed current}  gives the desired Sobolev inequalities for $T$.
\end{proof}

We now prove Theorem~\ref{theorem diameter and non collap for big class}.

\begin{proof}[Proof of Theorem~\ref{theorem diameter and non collap for big class}] Since the non-collapsing of volume implies the diameter bound, we only need to prove this part.
Let $T \in \mathcal{W}_{\text{big}}(X,\omega_X,n,p,A,B)$. 
Let $q \in (1, \frac{n}{n-1})$.  
Let $r \in (0,1]$ be a constant. Let $\rho: \R \to \R_{\ge 0}$ be a smooth function such that $\rho= 1$ on $[-1,1]$, $\rho=0$ outside $[-2,2]$ and $0 \le \rho \le 1$. 
Fix $x \in U_T$. We put $u:= d_T(x,\cdot)$ and  $u_r:= \rho(d_T(x,\cdot)/r)$. Observe that  $0 \le u_r \le 1$ and 
\begin{align}\label{ine-dur}
d u_r \wedge \dc u_r \le D r^{-2} du \wedge \dc u \le D r^{-2} T
\end{align}
 for some constant $D\ge 1$ depending only on $\rho$. Let $s \ge 2$ be a constant. Using (\ref{ine-dur}) and Theorem \ref{theorem Sobolev for big case}, we see that 
\begin{align}\label{ine-sobolev-def0hai3noncollap}
\bigg(\frac{1}{V_{T}}\int_X u_r^{sq} T^n\bigg)^{\frac{1}{q}} \le C s^{2} r^{-2}  \frac{1}{V_{T}} \int_X u_r^{s-2} T^n
\end{align} 
for $C = C(\omega_X,n,p,A,B,q)$. Here, we use  $0 \le u_r \le 1$. Choosing $s=3$ in \eqref{ine-sobolev-def0hai3noncollap} gives
\begin{align*}  
\bigg(\frac{1}{V_{T}}\int_X u_r^{3q} T^n\bigg)^{\frac{1}{q}} \le C r^{-2} \frac{1}{V_{T}} \int_X u_r T^n, 
 \end{align*}
By this and the definition of $u_r$, we infer that  
 $$\bigg(\frac{\vol_T(B_T(x,r/2))}{V_T}\bigg)^{1/q} \le C\frac{\vol_T(B_T(x,r))}{V_T r^{2}}$$ 
Now, arguing as in the proof of  \cite[Proposition 9.1]{Guo-Phong-Song-Sturm2}, we obtain 
$$\frac{\vol_T(B_T(x,r))}{V_T r^{\frac{2 q}{q-1}}} \ge C$$
as desired. This finishes the proof.
\end{proof}


In the last part of this section, we prove Theorems~\ref{theorem geometric for family}. 
Indeed, all of the previous estimates are not depending on the choice of $\omega_X$ except Lemma~\ref{lemma iii implies 1}. So, it suffices to give a family version of this lemma.

\begin{lemma}\label{lemma Lm bound for family}
    Let $p>1,A\geq 1,B\geq 1$ be positive constants. For every $m\geq 1$, $t\in \mathbb{D}^*$, and $T_t \in \mathcal{W}_{\text{big}}(X_t,\omega_t,n,p,A,B)$, there exists a constant $C = C(\omega_\mathcal{X},n,p,B,m)$ such that
    \[\frac{1}{V_{T_t}} \int_{X_t} (-\psi_t)^m T_t^n\leq C\]
    for every $\omega_t$-psh function $\psi_t$ with $\sup_{X_t} \psi_t= 0$.
\end{lemma}

\begin{proof}
  By \cite[Theorem 3.4]{DiNezzaGG} (we note that since the fibers $X_t$ are smooth for $t\neq 0$, our setting satisfies \cite[Assumption 3.2.3]{DiNezzaGG}), there exist constants $c = c(\omega_{\mathcal{X}})$ and $C = C(c,\omega_{\mathcal{X}})$ such that
    \[\frac{1}{V_{\omega_t}} \int_{X_t} e^{-c \psi_t} \omega_t^n \leq C\]
    for every $\omega_t$-psh function $\psi_t$ with $\sup_{X_t} \psi_t= 0$. This combined with the uniform $L^p(\omega_t^n)$-bound for the density of $V_T^{-1} T^n$ gives the desired estimate. 
\end{proof}

\begin{proof}[Proof of Theorem~\ref{theorem geometric for family}]
The proof of Theorem~\ref{theorem diameter and non collap for big class} applies here almost verbatim, the only change is that we use Lemma~\ref{lemma Lm bound for family} instead of Lemma~\ref{lemma iii implies 1}. 
\end{proof}

\section{\texorpdfstring{$L^1\log L^p$}--condition}\label{section entropy bound}
In this section, we will explain how to extend our method to the case of $L^1 \log L^p$-condition.
We first recall the following observation which was proved in \cite[Section 2.2]{Guedj-Lu-1}.
\begin{lemma}\label{lemma Lm integrable and entropy bound}\cite[Corollary 2.2]{Guedj-Lu-1} Let $(X,\omega_X)$ be a compact K\"ahler manifold. Let $p>n,A > 0$ be constants. Let $\mu = f \omega_X^n$ for $f\geq 0$ and $\int_X f |\log(f)|^p \omega_X^n \leq A$. For every $\omega_X$-psh function $\psi$ with $\sup_X \psi = 0$, we have
\[\int_X (-\psi)^p d\mu \leq C = C(\omega_X,p,A).\]
\end{lemma}

Let $p>n,A,B$ be positive constants. Let $\mathcal{W}^*_{\text{big}}(X,\omega_X,n,p,A,B)$ be the set of closed positive $(1,1)$-currents $T$ satisfying the following conditions:
\begin{itemize}
    \item[(i)] $T$ is a smooth K\"ahler form on an open Zariski dense subset of $X$, $T$ belongs to a big cohomology class and $T$ has minimal singularities;
    \item[(ii)] We have $\{T\} \cdot \{\omega_X\}^{n-1} \leq A\omega_X$, where $\{T\}$ denotes the cohomology class of $T$;
    \item[(iii)] $V_T^{-1} T^n = f_T \omega_X^n$ for some function $f_T$ such that $\|f_T\|_{L^1 \log L^p(\omega_X^n)} \leq B$.
\end{itemize}
\begin{theorem}\label{theorem geometric estimate entropy bound}
    Let $(X,\omega_X)$ be a compact K\"ahler manifold of dimension $n \geq 2$. Let $A , B$ be positive constants and $p > n$ be a constant. Let $q\in \Big(1,\frac{pn}{pn-p+n} \Big)$. 
    Then there exist positive constants $C_1= C_1(\omega_X,n,p,A,B)$ and $C_2 = C_2(\omega_X,n,p,A,B,q)$ such that for every $T\in \mathcal{W}_{\text{big}}^*
    (X,\omega_X,n,p,A,B)$ we have
\begin{itemize}
        \item[(i)] 
        \[\Big(\frac{1}{V_T} \int_{U_T} |u - \overline{u}|^{2q} T^n \Big)^{\frac{1}{q}} \leq C_2 \Big(\frac{1}{V_T} \int_{U_T} du\wedge d^c u \wedge T ^{n-1}\Big)\]
        for every $u\in W^{1,2}_T(X)$, where $\overline{u}:= V_T^{-1} \int_{U_T} u T^n$,
        \item[(ii)] 
    \[\Big(\frac{1}{V_T} \int_{U_T} |u|^{2q} T^n \Big)^{\frac{1}{r}} \leq C_2 \Big( \frac{1}{V_T} \int_{U_T} du\wedge d^c u \wedge T^{n-1} + \frac{1}{V_T} \int_{U_T} |u|^2 T^n\Big)\]
        for every $u\in W^{1,2}_T(X)$,
        \item[(iii)] $$\diam(\widehat{X},d_T) \leq C_1 \qquad \text{and} \qquad \frac{\vol_T(B_T(x,r))}{V_T} \geq C_2r^{\frac{2q}{q-1}}$$
    for every $r\in (0,\diam(\widehat{X},d_T)]$ and $x\in \widehat{X}$. 
    \end{itemize}

\end{theorem}
\begin{proof}
    The proof of Theorem~\ref{theorem diameter and non collap for big class} and Theorem~\ref{theorem Sobolev for big case} applies here almost verbatim. We only mention different points. Firstly, we replace Lemma~\ref{lemma iii implies 1} by Lemma~\ref{lemma Lm integrable and entropy bound}. Next, we note that the condition~\eqref{eq cond for m} holds for every $m>n$ provided that $n\geq 2$. Finally, the range for $q$  is $(1,\frac{pn}{pn -p +n})$  because we can only choose $m = p$ in Lemma~\ref{lemma Lm integrable and entropy bound} instead of any $m \geq 1$ in Lemma~\ref{lemma iii implies 1}.
\end{proof}

\bibliography{biblio_family_MA,biblio_Viet_papers,bib-kahlerRicci-flow}
\bibliographystyle{alpha}

\bigskip

\noindent
\Addresses

\end{document}